\documentclass[twoside,11pt]{amsart}

\usepackage[english]{babel}
\usepackage[utf8x]{inputenc}
\usepackage{url}
\usepackage{amsmath}
\usepackage{graphicx}
\usepackage[colorinlistoftodos]{todonotes}
\usepackage{amsmath,amscd,amsthm}
\usepackage{amstext, amsfonts, a4}
\usepackage{amssymb}
\usepackage{tikz-cd}
\usepackage{fancyhdr}
\usepackage{enumerate}
\usepackage{cleveref}
\usepackage{xcolor}

\numberwithin{equation}{section}
\newtheorem{thm}[equation]{Theorem}

\newtheorem{prop}[equation]{Proposition}
\newtheorem{cor}[equation]{Corollary}
\newtheorem{lem}[equation]{Lemma}

\theoremstyle{definition}

\newtheorem{ex}[equation]{Example}

\newtheorem{qu}[equation]{Question}

\newtheorem{rem}[equation]{Remark}

\renewcommand{\dim}{\operatorname{\mathsf{dim}}}
\renewcommand{\deg}{\operatorname{\mathsf{deg}}}
\renewcommand{\bmod}{\operatorname{\,\,\mathsf{mod}}\,\,}
\newcommand\ind{\operatorname{\mathsf{ind}}}

\newcommand\op{\operatorname{\mathsf{op}}}
\newcommand\End{\operatorname{\mathsf{End}}}
\newcommand\Br{\operatorname{\mathsf{Br}}}

\newcommand\id{\operatorname{\mathsf{id}}}

\newcommand\Nrd{\operatorname{\mathsf{Nrd}}}
\newcommand\Trd{\operatorname{\mathsf{Trd}}}
\newcommand\Tr{\operatorname{\mathsf{T}}}

\newcommand\hh{\operatorname{\mathbb{H}}}
\newcommand\N{\operatorname{N}}

\newcommand{\can}{\operatorname{\mathsf{can}}}

\newcommand{\vf}{\varphi}
\newcommand{\mg}[1]{{#1}^{\times}}
\newcommand{\sq}[1]{{#1}^{\times 2}}
\newcommand{\scg}[1]{\mg{#1}/\sq{#1}}
\newcommand{\s}{\sigma}
\newcommand{\nat}{\mathbb{N}}

\newcommand{\la}{\langle}
\newcommand{\ra}{\rangle}
\newcommand{\lla}{\la\!\la}
\newcommand{\rra}{\ra\!\ra}      

\newcommand{\an}{\mathsf{an}}

\renewcommand{\leq}{\leqslant}
\renewcommand{\geq}{\geqslant}

\newcommand{\tors}{\mathsf{tors}}

\newcommand\Ad{\operatorname{\mathsf{Ad}}}
\newcommand\ad{\operatorname{\mathsf{ad}}}

\newcommand{\W}{\mathsf{W}}
\newcommand{\I}{\mathsf{I}}
\newcommand{\rr}{\mathbb{R}}

\newcommand{\mc}{\mathcal}
\renewcommand{\N}{\mathsf{N}}

\newcommand{\disc}{\mathsf{disc}}
\newcommand{\Sim}{{\bf\mathsf{Sim}}}
\newcommand{\PSim}{{\bf\mathsf{PSim}}}
\newcommand{\G}{\mathsf{G}}
\renewcommand{\H}{\mathsf{H}}

\newcommand{\Hyp}{\mathsf{Hyp}}

\renewcommand{\setminus}{\smallsetminus}

\renewcommand{\deg}{\mathsf{deg}}
\renewcommand{\dim}{\mathsf{dim}}
\renewcommand{\leq}{\leqslant}
\renewcommand{\geq}{\geqslant}
\renewcommand{\sup}{\mathsf{sup}}

\newcommand{\D}{\mathsf{D}}
\newcommand{\Cl}{\mc{C}}

\newcommand{\wi}{\mathsf{i}}
\newcommand{\mf}{\mathfrak}

\newcommand{\zz}{\mathbb{Z}}
\newcommand{\cc}{\mathbb{C}}
\newcommand{\qq}{\mathbb{Q}}

\newcommand{\cd}{\mathsf{cd}}
\newcommand{\bil}{\mathsf{\mf{b}}}
\newcommand{\X}{\mathsf{X}}
\newcommand{\sos}[1]{\mathsf{\Sigma}{#1}^2}

\newcommand{\sign}{\mathsf{sign}}

\renewcommand{\bmod}{\,\mathsf{mod}\,}

\renewcommand{\setminus}{\smallsetminus}

\makeatletter
\newcommand{\bigperp}{%
  \mathop{\mathpalette\bigp@rp\relax}%
  \displaylimits
}

\newcommand{\bigp@rp}[2]{%
  \vcenter{
    \m@th\hbox{\scalebox{\ifx#1\displaystyle2.1\else1.5\fi}{$#1\perp$}}
  }%
}
\makeatother

\title{Similitudes over fields with $\I^4 =0$}

\date{25 February, 2026}

\author{M.~Archita}
\author{Karim Johannes Becher}

\address{University of Antwerp, Department of Mathematics, Antwerp, Belgium.}
\email{karimjohannes.becher@uantwerpen.be}
\email{archita.mondal@uantwerpen.be}

\begin{document}

\begin{abstract}
This article studies the set of $R$-equivalence classes of the group of proper projective similitudes of an algebra with involution of the first kind.
The main results concern base fields of characteristic different from $2$ over which every $9$-dimensional quadratic form has a nontrivial zero.
This includes function fields of $p$-adic curves and extensions of transcendence degree $3$ of $\cc$.
Main results of \cite{PS15} and \cite{PS17} are extended by relaxing the condition on the base field as well as on the Clifford invariant for orthogonal involutions.

\medskip\noindent
{\sc Keywords:} 
Classical adjoint algebraic group, $R$-equivalence, stably rational, algebra with involution, orthogonal, symplectic, hyperbolic, quadratic form, Pfister form, similarity factor, $2$-extension, reduced norm, cohomological dimension, $u$-invariant, field ordering, signature, weakly hyperbolic

\medskip\noindent
{\sc Classification (MSC 2020):} 11E04, 
11E57, 
11E81, 
14E08, 
20G15 
\end{abstract}

\maketitle

\section{Introduction} 

In this article we study adjoint semisimple classical linear algebraic groups of type $C$ and $D$ over a field $K$ of characteristic different from $2$. 
In view of Weil's classification results \cite{Wei61}, such a group is given as ${\bf PSim}^+(A,\s)$, the connected component of the identity in the group projective similitudes ${\bf PSim}(A,\s)$, for some central simple $K$-algebra $A$ of even degree and a $K$-linear involution $\s$ on $A$ which is symplectic (type $C$) or orthogonal (type $D$).

A finite field extension $L/K$ is called a \emph{$2$-extension} if $[L:K]=2^r$ for some $r\in\nat$ and there exists a sequence of intermediate fields $(K_i)_{i=0}^r$ such that $K_0=K$, $K_r=L$ and $K_i/K_{i-1}$ is a quadratic extension for $1\leq i\leq r$.

At the center of our work is the following general question.

\begin{qu}\label{Q:Merk-2-ext}
Which conditions on $(A,\s)$ and on $K$ imply that the group of similarity factors $\G(A,\s)$ is generated by $\sq{K}$ and the elements of $\mg{K}$ which are norms from finite $2$-extensions $L/K$ such that $\s_L$ is hyperbolic?
\end{qu}

By Merkurjev's criterion from \cite[Theorem~1]{Mer96}, a positive answer to \Cref{Q:Merk-2-ext} implies that the group of $K$-rational $R$-equivalence classes $${\bf PSim}^+(A,\s)(K)/R$$ is trivial.
See Section 3 for details.

The focus on $2$-extensions in  \Cref{Q:Merk-2-ext} is a central feature of this article.
Merkurjev's characterization does not impose this. 
It is motivated by the fact that we rely crucially on tools from classical quadratic form theory, where there is more control on $2$-extensions than on arbitrary finite extensions.
In this context we show that various results known under a given bound on the cohomological $2$-dimension remain valid with a milder hypothesis in terms of Pfister forms.

We denote by $\W K$ the Witt ring of quadratic forms over $K$ and by $\I K$ its fundamental ideal. 
For $n\in\nat$, we denote $\I^n K=(\I K)^n$, the $n$th power of $\I K$,  generated by the Witt equivalence classes of $n$-fold Pfister  forms over $K$.

Our main results concern fields $K$ where $\I^3K(\sqrt{-1})=0$ or 
$\I^4K= 0$.
In either of these situations, we achieve refinements and extensions of existing results. 

We will treat central simple algebras with involutions of the first kind (i.e.~trivial on the center). 
For orthogonal involutions, we restrict our attention to those of trivial discriminant, and we look at conditions on their Clifford algebra.
The behavior of orthogonal involutions of trivial Clifford invariant resembles that of symplectic involutions in that they have trivial cohomological invariants in degree $1$ and $2$.

In Section 5, we revisit positive results  \cite{KP08} and \cite{AP22} for orthogonal and symplectic involutions regarding \Cref{Q:Merk-2-ext}, assuming only that $\I^3K(\sqrt{-1})=0$.

In Section 6, we obtain some partial analogues under the assumption that $\I^{n+1}K(\sqrt{-1})=0$, $\ind A\leq 2$ and $(A,\s)$ has trivial cohomological invariants in degree less than $n$, for  some arbitrary $n\in\nat$.

For a $K$-quaternion algebra $Q$, we denote by $K_Q$ the function field of the projective conic $\mc{C}_Q$ associated to $Q$.
If $Q= (a,b)_K$ with $a,b\in\mg{K}$, then $\mc{C}_Q$ is given by $aX^2+bY^2=Z^2$.
When $A$ is Brauer equivalent to $Q$ and $\s$ is orthogonal, there is much control on $\s$ under scalar extension to $K_Q$.
Using this to reduce to the split case, we show in \Cref{ind2-main} that, whenever $\I^{n+1}K(\sqrt{-1})=0$, $A$ splits over every real closure of $K$ and $\ind A\leq 2$, then the answer to \Cref{Q:Merk-2-ext} is positive for $K$-linear involutions on $A$ having trivial cohomological invariants up to degree $n-1$.
It remains yet to be investigated whether this positive answer can be extended  by relaxing the conditions on $A$.

In Section 7, we study the case where $\I^4 K =0$.

\begin{qu}\label{Q:I4}
Assume that $\I^4 K=0$. If $\s$ is orthogonal, then assume that it has trivial discriminant.
Does it follow that ${\bf PSim}^+(A,\s)(K)/R=1$?
\end{qu}

In \Cref{u8-trivial-clifford}, we show that the answer is positive whenever $K$ has $u$-invariant at most $8$ and  $\s$ is  symplectic or has trivial Clifford invariant.
This was previously shown \cite{PS15} and \cite{PS17} in the special case where $K$ is the function field of a curve over a $p$-adic number field $\qq_p$ for an odd prime  $p$.
In \Cref{I4=0-Cliffindex2},
only assuming that $\I^4K=0$, we obtain a positive answer to \Cref{Q:I4} when $\s$ is orthogonal and both components of the Clifford algebra $\Cl(A,\s)$ have index at most $2$ (which implies that $\ind A\leq 4$).

In our exposition, we strive to present cases of orthogonal and symplectic involutions together whenever the same conclusion can be taken. We further circumvent case distinctions according to the parity of the rank of an associated hermitian form, which occur in \cite{PS15} and \cite{PS17}.
\newpage

\section{Quadratic forms and cohomological invariants} 

We follow standard terminology and notation from quadratic form theory, taking \cite{EKM08} and \cite{Lam05} as our main references.
By a \emph{quadratic form} we always mean a regular quadratic form.

Let $K$ be a field of characteristic different from $2$.
The nonzero squares in $K$ form the subgroup $\sq{K}$ of the multiplicative group $\mg{K}$, giving rise to square class group $\scg{K}$.

Consider a quadratic form $\vf$ over $K$.
It is considered as a homogeneous quadratic polynomial in a given number $n\in\nat$ of variables, where $n$ is called the \emph{dimension of $\vf$} and denoted $\dim(\vf)$. It gives rise to a quadratic map $K^n\to K$.
For a field extension $L/K$, we denote by $\vf_L$  the same polynomial considered with its evaluation map $L^n\to L$.
For $k\in\nat$, we denote by $k\times \vf$ the $k$-fold multiple $\vf\perp\dots\perp\vf$.
We denote by $\vf_\an$ the \emph{anisotropic part} and by $\wi(\vf)$ the 
\emph{Witt index} of $\vf$.
They are determined by having that $\vf_\an$ is anisotropic and 
$\vf\simeq \vf_\an\perp \wi(\vf)\times \hh$, where $\hh$ denotes the hyperbolic plane over $K$.
We call $\vf$ \emph{isotropic} if $\wi(\vf)>0$, and \emph{anisotropic} otherwise.
We call $\vf$ \emph{hyperbolic} if $\vf_\an$ is trivial, i.e. if $\vf\simeq i\times \hh$ for $i=\wi(\vf)$.
We denote by $\disc(\vf)$ the discriminant of $\vf$, considered as a class in $\scg{K}$.
We say that $\disc(\vf)$ is \emph{trivial} if $1\in\disc(\vf)$, that is, if $\disc(\vf)=1\sq{K}$. 
We denote by $\D(\vf)$ the set of elements of $\mg{K}$ that are represented by $\vf$, and we set 
$\G(\vf)=\{a\in\mg{K}\mid a\vf\simeq \vf\}$, which is a subgroup of $\mg{K}$.
If $\dim(\vf)=2$ we also call $\vf$ a \emph{binary form}.
For an ideal $J$ of the Witt ring $\W K$, we  write $\vf\in J$ to indicate that the Witt equivalence class of $\vf$ belongs to $J$.

For a central simple $K$-algebra $A$, we denote by $\deg A$  its \emph{degree} and by $\ind A$  its \emph{index}, and we call $A$ \emph{split} if $A\sim K$, or equivalently, if $\ind A=1$.
We denote by $\Br(K)$ the Brauer group of $K$ and by $\Br_2(K)$ its $2$-torsion part.
The elements of $\Br(K)$ are the classes $[A]$ of central simple $K$-algebras $A$.
The tensor product induces the group operation in $\Br(K)$, which we write additively, and we denote with neutral element $0=[K]$.
Given two central simple $K$-algebras $A$ and $A'$, we write $A\sim A'$ to indicate that $[A]=[A']$.

For the definition and the basic properties of the Clifford invariant of quadratic forms,  we refer to \cite[Chap.~V]{Lam05}. 
Let us recall what we shall need.
The Clifford invariant is a map
$$c : \W K \to \Br_2(K).$$
If $\vf$ is even-dimensional, then its Clifford algebra $\Cl(\vf)$ is a central simple $K$-algebra and $c(\vf)=[\Cl(\vf)]$.

\begin{thm}[Merkurjev]\label{Merkurjev}
      The restriction of $c$ to $\I^2 K$ gives a surjective group homomorphism $c:\I^2K\rightarrow \Br_2(K)$ whose its kernel is equal to $\I^3K$. 
\end{thm}
\begin{proof}
    See \cite{Mer81}.
\end{proof}

For $k\in\nat$, we denote by $\mu_k$ the group of $k$th roots of unity in an algebraic closure of $K$. 
Let $n\in\nat$. 
For an abelian group $A$ endowed with an action of the absolute Galois group $\mc{G}_K$ of $K$, we denote by $H^n(K,A)$ the $n$th cohomology group $H^n(\mc{G}_K,A)$.
Here, we consider the natural action of $\mc{G}_K$ on $\mu_k^{\otimes n-1}$.

For $n\in\nat$ and $a_1,\dots,a_n\in\mg{K}$, we denote by $\lla a_1,\dots,a_n\rra$ the $n$-fold Pfister form $\la 1,-a_1\ra\otimes\dots\otimes\la 1,-a_n\ra$ over $K$.

\begin{thm}[Orlov-Vishik-Voevodsky]
\label{OVV}
There is a natural surjective group homomorphism
$$e_n: \I^n K \to H^n(K,\mu_2)$$ 
whose kernel is equal to $\I^{n+1}K$ and
which maps the class of $\lla a_1,\dots,a_n\rra$ to the symbol $(a_1) \cup \cdots \cup (a_n)$, for any $a_1, \ldots, a_n \in \mg{K}$.
\end{thm}

\begin{proof}
See \cite{Voe03} in combination with \cite[Theorem~4.1]{OVV07} or \cite[Theorem 1.1]{Mor05}.
\end{proof}

Note that $H^1(K,\mu_2)$ is naturally isomorphic to $\scg{K}$, and this isomorphism makes $e_1$ correspond to the discriminant.
Furthermore,
$H^2(K,\mu_2)$ can be naturally identified with $\Br_2(K)$ (cf.~\cite[Cor.~4.4.5]{GS17}), and this identifies $e_2$ with the map $c$ of \Cref{Merkurjev}.
We write $\nat^+$ for $\nat\setminus\{0\}$.

We denote by $\X (K)$ the set of all field orderings of $K$.
For $P\in\X (K)$, we denote by $K_P$ the corresponding real closure of $K$.
For a quadratic form $\vf$ over $K$, we denote by $\sign_P(\vf)$ its signature at an ordering $P\in\X(K)$.
We call the field $K$ \emph{real} if $\X(K)\neq \emptyset$, and \emph{nonreal} otherwise.

We denote by $\sos{K}$ the set of sums of squares in $K$.
By the Artin-Schreier Theorem, \cite[Chap.~VIII, Theorem 1.10]{Lam05}, $K$ is real if and only if $-1\notin\sos{K}$.

\begin{cor}\label{2tors-cohom}
Let $n\in\nat^+$ be such that $\I^n K(\sqrt{-1})=0$.
Then the natural homomorphism $H^n(K,\mu_2)\to\prod_{P\in\X (K)} H^n(K,\mu_{2})$ is injective. Furthermore, for any $r,d\in\nat$, we have $H^n(K,\mu_{2^r}^{\otimes d})= H^{n}(K,\mu_2)=(-1)\cup H^{n-1}(K,\mu_2)$.
\end{cor}
\begin{proof}
Since $\I^nK(\sqrt{-1})=0$, we have that $\I^nK=2\times\I^{n-1}K$, by \cite[Cor. 35.27]{EKM08}.
Hence, using the surjectivity of $e_n:\I^nK\to H^n(K,\mu_2)$, we obtain that  $H^{n}(K,\mu_2)=(-1)\cup H^{n-1}(K,\mu_2)$.

Let $r,d\in\nat$.
Consider $\xi\in\H^n(K,\mu_{2^r}^{\otimes d})$ and assume that $\xi\neq 0$.
Let $i\in\nat$ be minimal such that $2^i\xi=0$.
Then $i>0$. Let $\xi'=2^{i-1}\xi$. Then $\xi'\neq 0$ and $2\xi'=0$.
Therefore $\xi'\in H^n(K,\mu_2)\setminus\{0\}$, and it follows by \Cref{OVV} that $\xi'=e_n(\vf)$ for a quadratic form $\vf\in\I^n K\setminus \I^{n+1}K$.
As $\I^nK(\sqrt{-1})=0$ and $\vf\notin \I^{n+1}K$, it follows by \cite[Cor.~35.27]{EKM08} together with
\cite[Theorem 1]{Kru90} that there exists  $P\in\X (K)$ such that  $\sign_{P}(\vf)\notin 2^{n+1}K$.
Then $\xi'_{K_P}=e_n(\vf_{K_P})\neq 0$, and consequently $\xi_{K_P}\neq 0$.
This proves that the natural homomorphism 
$H^n(K,\mu_{2^r}^{\otimes d})\to \prod_{P\in\X (K)} H^n(K_P,\mu_{2^r}^{\otimes d})$ is injective.
Since for any real closed field $R$ we have $H^n(R,\mu_{2^r}^{\otimes d})=H^n(R,\mu_2)\simeq \zz/2\zz$, the statement follows.
\end{proof}

As a consequence of \Cref{OVV}, the cohomological $2$-dimension of $K$, $\cd_2(K)$, can be characterised as follows (see e.g.~\cite[3.4. Remark]{BDGMZ}):  
$$\cd_2(K)=\sup\{n\in\nat\mid \I^nK'\neq 0\mbox{ for some finite field extension }K'/K\}\,.$$

\begin{prop}\label{In-vanish2ext}
    Let $n\in\nat$ be such that $\I^n K=0$.
    Then $\I^nL=0$ for every finite $2$-extension $L/K$.
\end{prop}
\begin{proof}
    See \cite[Cor.~35.8]{EKM08} for the case where $[L:K]=2$. The general case follows  immediatley by induction on  $r\in\nat$ such that $[L:K]=2^r$.
\end{proof}

\section{Algebras with involution, similitudes and norms} 

Our main reference for algebras with involution is \cite{KMRT98}.
For terminology and basic facts concerning central simple algebras, we refer to \cite{GS17} and \cite{Pie82}.

Let $A$ be a central simple $K$-algebra and let $\s$ be an involution on $A$.
We call $\s$ \emph{hyperbolic} if there exists an element $e\in A$ such that $e^2=e$ and $\s(e)=1-e$.

We restrict our attention to involutions \emph{of the first kind}, that is, which are $K$-linear.
By a \emph{$K$-algebra with involution} we mean a pair $(A,\s)$ where $A$ is a central simple $K$-algebra and $\s$ is a $K$-linear involution on $A$. 
Such involutions are distinguished into \emph{orthogonal} and \emph{symplectic} ones, see \cite[\S2]{KMRT98}.

Let $(A,\s)$ be a $K$-algebra with involution. 
We call $(A,\s)$ \emph{hyperbolic} if $\s$ is hyperbolic.
This can only occur when $\deg A$ is a multiple of $2\ind A$, so in particular $\deg A$ must be even.
We call $(A,\s)$ \emph{split hyperbolic} if $A$ is split and $\s$ is hyperbolic.
For a field extension $L/K$, we obtain an $L$-algebra with involution $(A,\s)_L=(A_L,\s_L)$, where $A_L=A\otimes_KL$ and $\s_L=\s\otimes\id_L$.

We denote $\Sim(A,\s)=\{x\in A\mid \s(x)x\in \mg{K}\}$ and obtain that this is a subgroup of $\mg{A}$. 
We 
denote by $\PSim(A,\s)$ the quotient group $\Sim(A,\s)/\mg{K}$.
The elements of $\Sim(A,\s)$ are called \emph{similitudes of $(A,\s)$}, those of $\PSim(A,\s)$ are called \emph{projective similitudes of $(A,\s)$}.

Letting 
$${\bf PSim}(A,\s)(K')=\PSim(A_{K'},\s_{K'})$$
for every field extension $K'/K$ defines an algebraic group over $K$ whose set of $K$-rational points is $\PSim(A,\s)$.
We denote by ${\bf PSim^{+}} (A, \sigma)$ be the connected component of the identity in 
${\bf PSim} (A,\sigma)$. 
This is a classical semisimple adjoint linear algebraic group.
We mention that ${\bf PSim}^{+}(A,\s)={\bf PSim}(A,\s)$ unless $\s$ is orthogonal and $\deg A$ is even, and in the latter case, ${\bf PSim}^+(A,\s)(K')$ has index $2$ in ${\bf PSim} (A,\s)(K')$ for every field extension $K'/K$. (See \cite[Prop. 12.23]{KMRT98}.)
We obtain a natural group homomorphism 
$$\mu:\Sim(A,\s)\to \mg{K}, x\mapsto \s(x)x\,.$$
Its image is denoted by $\G(A,\s)$. In other terms,
 $$\G(A,\s)=\{\s(x)x\mid x\in\Sim(A,\s)\}\,.$$ 
 The elements of this subgroup of $\mg{K}$ are called \emph{similarity factors of $\s$.}
We further set
$$\G^+(A,\s)=\{\mu(a)\mid a\mg{K}\in{\bf PSim}^+(A,\s)(K)\}\,.$$
This is a subgroup of $\G(A,\s)$ of index at most $2$.

Consider a quadratic form over $K$, given on the $K$-vector space $V_\vf$.
The split central simple $K$-algebra $\End_K(V_\vf)$ then is endowed with a $K$-linear involution $\ad_\vf$, called the \emph{adjoint involution on $\vf$}, which is characterized by the property that $\bil_\vf(x,f(y))=\bil_\vf(\ad_\vf(f)(x),y)$ for every $f\in\End_K(V_\vf)$ and every $x,y\in V_\vf$, where $\bil_\vf:V_\vf\times V_\vf\to K$ is the symmetric $K$-bilinear form associated to $\vf$ given by $\bil_\vf(x,y)=\vf(x+y)-\vf(x)-\vf(y)$.
We set
$$\Ad(\vf)=(\End_K(V_\vf),\ad_\vf)\,.$$
This is a split $K$-algebra with orthogonal involution.
By \cite[p.1, Theorem]{KMRT98}, every split $K$-algebra with orthogonal involution is isomorphic to $\Ad(\vf)$ for a quadratic form $\vf$ over $K$,
 and 
$\Ad(\vf)$ is hyperbolic if and only if the quadratic form $\vf$ is hyperbolic in the usual sense, by \cite[Prop.~6.7]{KMRT98}.

\begin{prop}\label{G-G+}
If $\s$ is orthogonal of nontrivial discriminant, $\ind A=2$ and $A_{K(\sqrt{d})}$ is split for $d\in\disc(\s)$, then $\G^+(A,\s)$ has index $2$ in $\G(A,\s)$.
In every other case $\G^+(A,\s)=\G(A,\s)$.
\end{prop}
\begin{proof}
See \cite[Prop.~12.23]{KMRT98} for the case where $\s$ is symplectic, \cite[Cor.~13.43]{KMRT98} for the case where $A$ is not split and $\s$ is orthogonal.
It remains to consider the case where $A$ is split and $\s$ is orthogonal. 
Then $(A,\s)\simeq\Ad(\vf)$ for a quadratic form $\vf$ over $K$.
If $\deg A$ is odd, then $\G(A,\s)=\G^+(A,\s)=\sq{K}$.
Assume that $\deg A$ is even. 
We choose a hyperplane reflection $a$ of $\vf$.
Then $a$ represents a class in ${\bf PSim}(A,\s)\setminus {\bf PSim}^{+}(A,\s)$.
Hence $\G(A,\s)=\G^+(A,\s)\cup \mu(a)\G^+(A,\s)$, and since $\deg A$ is even, we have $\mu(a)=1$.
\end{proof}

For a finite field extension $L/K$, we denote by $\N_{L/K}:L\to K$ the norm map, and we fix the notation $$\N_{L/K}^\ast=\N_{L/K}(\mg{L})\,,$$ and note that this is a subgroup of $\mg{K}$.

For a subset $S\subseteq\mg{K}$, we denote by $\la S\ra$ the subgroup of $\mg{K}$ generated by $S$.

Let $\Phi$ be a quadratic form over $K$ or a $K$-algebra with involution.
Hence, for any field extension $K'/K$, we have an object $\Phi_{K'}$ obtained by scalar extension, and we can consider whether it is hyperbolic.

We denote by $\Hyp^\ast(\Phi)$
(resp. by $\Hyp_2^\ast(\Phi)$) the union of the subgroups $\N_{L/K}^\ast$ where $L/K$ is a finite field extension (resp.~a finite $2$-extension) such that $\Phi_L$ is hyperbolic.
We further set $\Hyp(\Phi)=\la \Hyp^\ast(\Phi)\ra$ and $\Hyp_2(\Phi)=\la\Hyp^\ast(\Phi)\ra$.
Note that $\Hyp_2(\Phi)\subseteq\Hyp(\Phi)$.
We do not know of an example where these two groups are different. 
Cases where $\Hyp^\ast(\Phi)\subsetneq\Hyp(\Phi)$ are easy to obtain, e.g.~from \Cref{G-in-Nrd}~$(a)$ below.
While the main attention lies on the groups $\Hyp(A,\s)$ and $\Hyp_2(A,\s)$, the $\ast$-notation will allow us to formulate more precise statements without becoming too cumbersome.

If $\Phi$ is a quadratic form of odd dimension over $K$ or (equivalently) a $K$-algebra with involution of odd degree,
then $\Hyp_2^\ast(\Phi)=\emptyset$ and $\Hyp_2(\Phi)=\{1\}$.

For a quadratic form $\vf$ over $K$, the adjoint $K$-algebra with involution $\Ad(\vf)$ is hyperbolic if and only if the quadratic form $\vf$ is hyperbolic, and thus we have $\Hyp_2^\ast(\Ad(\vf))=\Hyp_2^\ast(\vf)$ and $\Hyp^\ast(\Ad(\vf))=\Hyp^\ast(\vf)$.

\begin{prop}[Norm Principle]\label{SNP}
    Let $\Phi$ be a quadratic form or a $K$-algebra with involution.
    Then $\sq{K}\Hyp(\Phi)\subseteq \G^+(\Phi)$.
\end{prop}
\begin{proof}
This is covered by \cite[Theorem 1]{Mer96}.
See also~\cite[Chap. VII, Cor. 4.4]{Lam05} for the case where $\Phi$ is a quadratic form.
\end{proof}

\begin{thm}[Merkurjev]
\label{Mer:T1}
Let $\Phi$ be a quadratic form or a $K$-algebra with involution. 
For any field extension $K'/K$, the homomorphism $\mu$ induces an isomorphism on the group of $K'$-rational $R$-equivalence classes 
$${\bf PSim}^+(\Phi)(K')/R\stackrel{\simeq}{\longrightarrow} \G^+(\Phi_{K'})/\sq{K}\Hyp(\Phi_{K'})\,.$$
\end{thm}

\begin{proof}
    See \cite[Theorem 1]{Mer96}.
\end{proof}

\section{Reduced norms} 

For a $K$-algebra with orthogonal involution $(A,\s)$ of trivial discriminant, we look into the general relations (without conditions on $K$)  between $\G(A,\s)$, $\Hyp(A,\s)$ as well as the groups of reduced norms of $A$ and of the two components of $\Cl(A,\s)$.

Let $A$ be a central simple $K$-algebra.
We denote by $\Trd_A:A\to K$ the reduced trace and by $\Nrd_A:A\to K$ the reduced norm map of $A$, and we set $\Nrd_A^\ast=\Nrd_A(\mg{A})$, which is a subgroup of $\mg{K}$. 
If $A$ is a $K$-quaternion algebra, then $\Nrd_A$ is a quadratic form over $K$, isometric to $\lla a,b\rra$ for any $a,b\in\mg{K}$ with $A\simeq (a,b)_K$.

\begin{prop}\label{Nrd*}
The following hold:
\begin{enumerate}[$(a)$]
    \item $\Nrd_{A}^\ast=\Nrd_B^\ast$ for every central simple $K$-algebra $B\sim A$.
    \item $\Nrd_{A}^\ast$ is the union of the subgroups $\N_{L/K}^\ast$ for finite field extensions $L/K$ such that $A_L$ is split. 
    \item If $B$ is a central simple $K$-algebra and every finite field extension splitting $A$ also splits $B$, then $\Nrd_A^\ast\subseteq \Nrd_B^\ast$.
\end{enumerate}
\end{prop}
 \begin{proof}
    See \cite[Lemma 2.6.4]{GS17} for $(a)$.
    Part $(b)$ follows from $(a)$, using that a finite field extension $L/K$ splits $A$ if and only if $A\sim B$ for a central simple $K$-algebra $B$ containing $L$ as a maximal commutative subring (cf.~\cite[Theorem 13.3]{Pie82}), and  in this case $\Nrd_B|_L=\N_{L/K}$.
    Parts $(c)$ is immediate from $(b)$.
 \end{proof}

Assume that $\deg A$ is even and consider an orthogonal involution $\s$ on $A$.
We denote by $\disc(\s)$ the discriminant of $\s$, given as a coset in $\scg{K}$.
We may write $d\in\disc(\s)$ to indicate that $d\in\mg{K}$ is such that $\disc(\s)=d\sq{K}$.
We denote by $\Cl(A,\s)$ the Clifford algebra of $(A,\s)$.
If $\s$ has trivial discriminant (i.e. $1\in\disc(\s)$), then $\Cl(A,\s)\simeq C_1\times C_2$ for two central simple $K$-algebras $C_1$ and $C_2$, which we call the \emph{components of $\Cl(A,\s)$}, and by \cite[Theorem 9.12]{KMRT98}, we have  in $\Br(K)$ that
$[C_1]+[C_2]=[A]$  if $\deg A\equiv 0\bmod 4$ and $2[C_1]=2[C_2]=[A]$ if $\deg A\equiv 2\bmod 4$.

\begin{prop}\label{G-in-Nrd}
    Let $(A,\s)$ be a $K$-algebra with orthogonal involution of trivial discriminant and let $C_1$ and $C_2$ be the two components of $\Cl(A,\s)$. Then:
\begin{enumerate}[$(a)$]
    \item $\Hyp(A,\s)\subseteq \Nrd_{C_1}^\ast\cdot \Nrd_{C_2}^\ast$.
    \item  If $\ind(C_i)\leq 2$ for $i\in\{1,2\}$, then $\G(A,\s)\subseteq \Nrd_{C_1}^\ast\cdot \Nrd_{C_2}^\ast$.
    \item  If $\deg A=4$, then 
    $\G(A,\s)=\Nrd_{C_1}^\ast\cdot \Nrd_{C_2}^\ast=\Hyp_2(A,\s)$.
\end{enumerate}
\end{prop} 

 \begin{proof}
$(a)$ Consider $a\in\Hyp^\ast(A,\s)$.
Then there exists a finite field extension $L/K$ such that $\s_L$ is hyperbolic and $a\in\N_{L/K}^\ast$.
It follows by \cite[Prop.~8.31]{KMRT98} that either $(C_1)_L$ or $(C_2)_L$ is split.
This yields that $\Hyp^\ast(A,\s)\subseteq \Nrd_{C_1}^\ast\cup\Nrd_{C_2}^\ast$, whereby $\Hyp(A,\s)\subseteq \Nrd_{C_1}^\ast\cdot\Nrd_{C_2}^\ast$.

$(b)$  By the hypothesis, for $i\in\{1,2\}$, we have $C_i\sim Q_i$ for a $K$-algebra $Q_i$, and we denote by $\pi_i$ its norm form, to obtain that $\Nrd_{C_i}^\ast=\Nrd_{Q_i}^\ast=\D(\pi_i)$.
Let $a\in\G(A,\s)$.
We consider the quadratic form $\varphi=\pi_1\perp-a\pi_2$ over $K$. 
By \cite[Theorem 9.15]{KMRT98}, we have  $\Cl(\varphi)\sim Q_1\otimes Q_2\sim A$.
If $A$ is split, then $Q_1\simeq Q_2$ and $(A,\s)\simeq \Ad(\rho)$ for a quadratic form $\rho$ over $K$ with $\Cl(\rho)\sim  Q_1\sim Q_2$, and it follows that $\G(A,\s)=\G(\rho)\subseteq \G(\pi_1)=\G(\pi_2)$, whereby $\varphi$ is hyperbolic.
In the general case, let $K_A/K$ denote the splitting field of the Severi-Brauer variety associated to $A$. 
We conclude from the split case that $\varphi_{K_A}$ is hyperbolic. 
By \cite[Théorème 4]{Lag96}, this implies that $\varphi$ is isotropic over $K$.
This yields that $a\in \D(\pi_1)\cdot\D(\pi_2)=\Nrd_{C_1}^\ast\cdot \Nrd_{C_2}^\ast$.

$(c)$ Assume that $\deg A=4$.
By \cite[Corollary 15.13]{KMRT98}, then $C_1$ and $C_2$ are $K$-quaternion algebras and $(A,\s)\simeq (C_1,\can_{C_1})\otimes (C_2,\can_{C_2})$, and by \cite[Corollary 15.13]{KMRT98}, we have $\G(A,\s)=\Nrd_{C_i}^\ast\cdot\Nrd_{C_2}^\ast$.
From $(a)$ we obtain that $\Hyp_2(A,\s)\subseteq \Nrd_{C_i}^\ast\cdot\Nrd_{C_2}^\ast$.
To show the converse inclusion, consider  $a=a_1\cdot a_2$  with  $a_1\in \Nrd_{C_1}^\ast$ and $a_2\in \Nrd_{C_2}^\ast$.
For $i\in\{1,2\}$, there exists a field extension $L_i/K$ with $[L_i:K]\leq 2$ such that $a_i\in\N_{L_i/K}^\ast$ and $(C_i)_{L_i}$ is split, whereby  $(\can_{C_i})_{L_i}$ and $\s_{L_i}$ are hyperbolic, and consequently $a_i\in\N_{L_i/K}^\ast\subseteq\Hyp_2^\ast(A,\s)$. 
Hence $a=a_1a_2\in\Hyp_2(A,\s)$.
\end{proof}

In the remainder of this section, we look at the case where $\deg A\equiv 2\bmod 4$. 

\begin{prop}\label{G-Nrd-deg2mod4}
Let $(A,\s)$ be a $K$-algebra with orthogonal involution with $\deg A\equiv 2\bmod 4$. Then:
\begin{enumerate}[$(a)$]
    \item $\G^+(A,\s)\subseteq \Nrd_A^\ast$.
    \item If $\disc(\s)$ is trivial discriminant and $C$ is a component of  $\Cl(A,\s)$, then $\Cl(A,\s)\simeq C\times C^{\op}$  and $C^{\otimes 2}\sim A$, and in particular $A$ is split or $\ind C\geq 4$.
\end{enumerate}
 \end{prop}
\begin{proof}
$(a)$   Write $\deg A=2n$ with $n\in\nat$ and $\lambda\in\G^+(A,\s)$.
Let $a \in \Sim^+(A,\s)$ be such that $\lambda =a\s(a)$.
Then $\Nrd_A(\s(a))=\Nrd_A(a)=+\lambda^n$. 
This implies that 
$$\lambda^{2n}=\Nrd_A(\lambda)=\Nrd_A(a)\Nrd_A(\s(a))=\Nrd_A(a)^2\,.$$
By the hypothesis, $n$ is odd.
We conclude that $\lambda\in \sq{K}\cdot \Nrd_A^\ast$.
Since $A$ carries a $K$-linear involution, $\ind A$ is a $2$-power.
Since $\ind A$ divides $\deg A=2n$, it follows that  $\ind A\leq 2$.
Hence $A\sim Q$ for a $K$-quaternion algebra $Q$.
Using \Cref{Nrd*}, we obtain that 
$\sq{K}\subseteq \Nrd_Q^\ast=\Nrd_A^\ast$.

$(b)$ Assume that $\disc(\s)$ is trivial.
Then $\mc{C}(A,\s)\simeq C_1\times C_2$ for two central simple $K$-algebras $C_1$ and $C_2$ of the same degree.
By \cite[Theorem 9.11]{KMRT98}, we have $[C_1]+[C_2]=0$ and $2[C_1]=2[C_2]=[A]$ in $\Br(K)$.
Hence $C_2\simeq C_1^{\op}$.
Since $2[A]=0$, it follows that either $[A]=0$ or 
the class given by $[C_i]$ has order $4$ in $\Br(K)$ for $i\in\{1,2\}$, and in that case $\ind C_i$ is a multiple of $4$.
\end{proof}

\begin{ex}\label{A3D3}
Consider the case where $(A,\s)$ is orthogonal of degree $6$ with trivial discriminant.
Then $\Cl(A,\s) \simeq C\times C^{\op}$ for a central simple $K$-algebra $C$ of degree $4$.
In this case it follows that
$\G(A,\s)=\sq{K}\Nrd_C^\ast$; see \cite[Prop.~6]{Mer96} and the references there, along with \Cref{G-G+}.
\end{ex}

Recall that, for a $K$-quaternion algebra $Q$, we denote by $K_Q$ the function field of the projective conic associated to $Q$.

\begin{thm}\label{deg2mod4Cliffind4Nrd}
Let $(A,\s)$ be a $K$-algebra with orthogonal involution of trivial discriminant and such that $\deg A\equiv 2\bmod 4$.
Let $C$ be one component of $\Cl(A,\s)$ and $Q$ the $K$-quaternion algebra such that $A\sim Q$. 
Assume that $\ind C_{K_Q}\leq 4$.
Then $$\G(A,\s)\subseteq \sq{K}\Nrd_C^\ast\,.$$
\end{thm}
\begin{proof}
Note that $\Cl(A,\s)\simeq C\times C^{\op}$.
Since $\ind C_{K_Q}\leq 4$, we may choose a central simple $K$-algebra $C'$ with $\deg C'=4$ such that either $C'\sim C$ or $C'\sim C\otimes Q$.
Then there exists a central simple $K$-algebra with orthogonal involution $(A',\s')$ such that $\Cl(A',\s')\simeq C'\times C'^\op$.
Note that $A'\sim C'^{\otimes 2}\sim C^{\otimes 2} \sim A\sim Q$.
There exist skew-hermitian forms $(V,h)$ and $(V',h')$ over $(Q,\can_Q)$ such that $(A,\s)\simeq \Ad(V,h)$ and $(A',\s')\simeq \Ad(V',h')$.
We obtain a skew-hermitian form $(V\oplus V',h\perp h')$ over $(Q,\can_Q)$ and consider its adjoint $K$-algebra with orthogonal involution $(A'',\s'')=\Ad(V\oplus V',h\perp h')$.
(This is an orthogonal sum of $(A,\s)$ and $(A',\s')$ in the sense of Dejaiffe \cite{Dej98}.)
Since $\disc(\s)$ is trivial, so is $\disc(\s')$, by \cite[Prop.~2.3]{Dej98}.
Furthermore $\Cl(A'',\s'')\simeq (C\otimes C', C^{\op}\otimes C'^{\op})$.
If $C'\sim C$, then $C\otimes C'\sim Q$, and if $C'\sim C\otimes Q$, then $C\otimes C'$ is split.
Therefore $\Cl(A''_{K_Q},\s''_{K_Q})$ is trivial.

Since $A'\sim A\sim Q$, we have that $A,A'$ and $A''$ split over $K_Q$.
Hence there exist quadratic forms $\vf$ and $\vf'$ over $K_Q$ such that $(A,\s)_{K_Q}\simeq \Ad(\vf)$ and $(A',\s')_{K_Q}\simeq \Ad(\vf')$.
It follows that $\Ad(\vf\perp\vf')\simeq (A'',\s'')_{K_Q}$.
Note that $\vf$ and $\vf'$ are even-dimension of trivial discriminant, that is, $\vf,\vf'\in\I^2K_Q$.
Furthermore, $c(\vf)=[C_{K_Q}]$ and $c(\vf')=[C'_{K_Q}]$.
Since $C_{K_Q}\sim C'_{K_Q}$, we conclude that $c(\vf\perp\vf')=0$. 
Hence $\vf\perp\vf'\in \I^3 K_Q$, by \Cref{Merkurjev}.

Consider now an element $a\in\G(A,\s)$.
Then $a\in\G(\vf)$, whereby $\lla a\rra\otimes \vf$ is hyperbolic.
Since $\lla a\rra\otimes (\vf\perp\vf')\in \I^4K_Q$, we conclude that $\lla a\rra\otimes\vf'\in \I^4K_Q$.
Since $\dim(\vf')=\deg A'=6$, we have $\dim(\lla a\rra\otimes \vf')=12$.
We conclude by the Arason-Pfister Hauptsatz \cite[Theorem 23.7]{EKM08} that $\lla a\rra\otimes \vf'$ is hyperbolic.
Hence $\left(\Ad\lla a\rra\otimes (A',\s')\right)_{K_Q}$ is hyperbolic.
Hence, Dejaiffe's Theorem \cite{Dej01} yields that $\Ad\lla a\rra\otimes (A',\s')$ is hyperbolic, whereby $a\in\G(A',\s')$.
This argument shows that $\G(A,\s)\subseteq G(A',\s')$.
Since $\deg A'=6$ and $\Cl(A',\s')=C'\times C'^{\op}$, we have $\G(A',\s')=\sq{K}\Nrd_{C'}^\ast$, by \Cref{A3D3}.
Note that $C$ and $C'$ are split over the same field extensions. Hence $\Nrd_{C'}^\ast=\Nrd_C^\ast$, by \Cref{Nrd*}, and consequently  
$\G(A,\s)\subseteq \G(A',\s')=\sq{K}\Nrd_{C'}^\ast=\sq{K}\Nrd_{C}^\ast$.
\end{proof}

\section{Norms from splitting $2$-extensions} 

We study the situation where $\I^3K(\sqrt{-1})=0$.
We revisit some results from \cite{KP08} and \cite{AP22}, replacing the hypothesis that $\cd_2(K(\sqrt{-1}))\leq 2$ there by the milder condition that $\I^3K(\sqrt{-1})=0$.

We denote by $\I_\tors K$ the torsion ideal of  $\W K$ and set $\I^n_\tors K=\I^n K\cap \I_\tors K$ for $n\in\nat$.

\begin{prop}\label{In-torsfree-torsion-In-Hyp2}
    Let $n\in\nat$ be such that $\I^{n+1}K(\sqrt{-1})=0$.
    Let $\vf\in\I^n_\tors K$. Then $\Hyp_2^\ast(\vf)=\mg{K}$.
\end{prop}
\begin{proof}
This is \cite[Prop.~4.1]{AB25}.
\end{proof}

A central theme of this and the next section is the quest for extensions  of \Cref{In-torsfree-torsion-In-Hyp2} to algebras with involution, replacing conditions such as being torsion accordingly, in particular for $n=2$.

Let $A$ be a central simple $K$-algebra.
We associate the quadratic form $$\Tr_A:A\to K,x\mapsto \Trd_A(x^2)$$ over $K$, called the \emph{trace form of $A$}.
For an ordering $P\in\X(K)$, one has 
$\sign_P(\Tr_A)=\pm\deg A$, or more precisely,
$$\sign_P(\Tr_A)=\begin{cases}
    +\deg(A) & \mbox{ if $A_{K_P}$ is split,}\\
    -\deg(A) & \mbox{ if $A_{K_P}\sim (-1,-1)_{K_P}$\,.}
\end{cases}
$$ 
We call the central simple $K$-algebra $A$ \emph{totally positive} if $\sign_P(\Tr_A)>0$ for all $P\in \X(K)$, or equivalently, if $A$ splits over every real closure of $K$.
Note that, if $K$ is nonreal, then every central simple $K$-algebra is totally positive.
(Without giving it a name, the condition of being totally positive has been considered in various preceding works, in particular \cite{BFP98}. In \cite{Bec08}, the term `nonreal' was used instead of `totally positive', but that leads to a conflict of terminology when applied to $A=K$.)

The following is a refinement of a  characterization of reduced norms given in \cite[Theorem 2.1]{BFP98} under the stronger assumption that $\cd_2(K(\sqrt{-1}))\leq 2$.
When $K$ is nonreal, the statement corresponds to \cite[Prop.~2.1]{KP08}.

\begin{prop}\label{BFP98-Thm2-1-general}
Assume that $\I^3 K(\sqrt{-1})=0$.
Let $A$ be a central simple $K$-algebra representing a $2$-primary torsion element of $\Br(K)$.
For $a\in\mg{K}$, the following are equivalent:
\begin{enumerate}[$(i)$]
    \item $a\in\Nrd_A^\ast$.
    \item $a\in P$ for all $P\in\X(K)$ with $\sign_P(\Tr_A)<0$.
    \item $a\in\N_{L/K}^\ast$ for a finite $2$-extension $L/K$ such that $A_L$ is split.
\end{enumerate}
In particular, if $A$ is totally positive, then $\Nrd_A^\ast=\mg{K}$.
\end{prop}
\begin{proof}
$(iii\Rightarrow i)$
This follows from \Cref{Nrd*}~$(b)$.

$(i\Rightarrow ii)$
Let $P\in\X(K)$ be an ordering such that $\sign_P(A)<0$. 
Then $A_{K_P}\sim (-1,-1)_{K_P}$, and hence $\Nrd_{A_{K_P}}^\ast=\Nrd_{(-1,-1)_{K_P}}^\ast=\sq{(K_P)}$, in view of \Cref{Nrd*}~$(a)$.
Therefore $\Nrd_{A}^\ast\subseteq\mg{K}\cap\sq{(K_P)}\subseteq P$.

$(ii\Rightarrow iii)$
Consider $a\in\mg{K}$ such that $a\in P$ for every $P\in\X(K)$ with $\sign_P(\Tr_A)<0$.
Let $K'=K(\sqrt{-a})$.
For $P'\in\X(K')$, letting $P=P'\cap K$, we have $P\in\X(K)$ and $a\notin P$, whereby $\sign_{P'}(\Tr_{A_{K'}})=\sign_P(\Tr_A)>0$.
Hence $A_{K'}$ is totally positive.

We proceed by induction on the smallest $r\in\nat$ such that $A^{\otimes 2^r}$ is split.
If $r=0$ then $A_{K'}$ is split and we may take $L=K'$.
Assume now that $r>0$.
The induction hypothesis applied to $A'=A^{\otimes 2^{r-1}}$ yields that $\sqrt{-a}=\N_{K''/K'}(c)$ for a finite $2$-extension $K''/K'$ such that $A'_{K''}$ is split and some $c\in\mg{K''}$. 
Then $K''/K$ is a finite $2$-extension, whereby 
$\I^3K''(\sqrt{-1})=0$, in view of \Cref{In-vanish2ext}.
Since $A'_{K''}$ is split,
 we have $[A_{K''}]\in\Br_2(K'')$.
 Moreover, 
$A_{K''}$ is totally positive.
It follows by \cite[Prop.~5.5]{Bec08} that $[A_{K''}]=c(\psi)$ for 
a quadratic form $\psi\in \I^2_\tors K''$.
By \Cref{In-torsfree-torsion-In-Hyp2}, $\Hyp_2^\ast(\psi)=\mg{K''}$, so there exists a finite $2$-extension $L/K''$ such that $c\in\N_{L/K''}^\ast$ and $\psi_L$ is hyperbolic.
Then $L/K$ is a finite $2$-extension with $[A_L]=c(\psi_L)=0$.
So $a=\N_{K'/K}(\sqrt{-a})=\N_{K''/K}(c)\in\N_{K''/K}(\N_{L/K}^\ast)=\N_{L/K}^\ast$ and $A_L$ is split.
\end{proof}

Consider a $K$-algebra with involution $(A,\s)$. 
If $\s$ is orthogonal, then we say $(A,\s)$ (or simply $\s$) \emph{has trivial Clifford invariant} if
$\deg A$ is even, $\disc(\s)$ is trivial and at least one of the two components of the Clifford algebra $\Cl(A,\s)$ is split.
We write $(A,\s)\in\I^2$ to indicate that $\deg A$ is even and either $\s$ is orthogonal with trivial discriminant or $\s$ is symplectic.
We write $(A,\s)\in\I^3$ to indicate that 
either $\s$ is orthogonal with trivial Clifford invariant or $\s$ is symplectic with $\deg A \equiv 0\bmod 4$.

We refer to \cite{KMRT98} and \cite{BU18} for the definition of signatures of an involution of the first kind.

Recall that $n\times \la 1\ra$ is the $n$-dimensional quadratic form $\la 1,\dots,1\ra$, for $n\in\nat$.
Let $\Phi$ be a $K$-algebra with involution.
We denote the $K$-algebra with involution $\Ad(n\times \la 1\ra)\otimes \Phi$ by $n\times \Phi$.
We call $\Phi$ \emph{weakly hyperbolic} if $n\times \Phi$ is hyperbolic for some $n\in\nat^+$.
Note that, if $K$ is nonreal, then every $K$-algebra with involution $\Phi$ is weakly hyperbolic, because for $r\in\nat$ such that $-1$ is a sum of $2^r$ squares in $K$, we have that $2^{r+1}\times \la 1\ra$ is hyperbolic, whereby $2^{r+1}\times \Phi$ is hyperbolic.
Weakly hyperbolic $K$-algebras with involution were characterized in \cite{LU03} and \cite{BU18}.

The following statement extends \cite[Theorem 4.5]{AP22} in the orthogonal case.
Note that the split case corresponds to \Cref{In-torsfree-torsion-In-Hyp2} for $n=2$.

\begin{prop}\label{KP}
Assume that $\I^3K(\sqrt{-1})=0$.
Let $(A,\s)$ be a $K$-algebra with involution.
Assume that $A$ is totally positive, $(A,\s)$ is weakly hyperbolic and $(A,\s)\in\I^2$.
Then
$$\Hyp_2^\ast(A,\s)=\G(A,\s)=\mg{K}\,.$$
\end{prop}

\begin{proof}
In view of \Cref{SNP}, we need only to show that $\mg{K}\subseteq\Hyp_2^\ast(A,\s)$.
Since $A$ carries a $K$-linear involution, we have $[A]\in\Br_2(K)$.

Consider $a\in\mg{K}$ arbitrary.
By \Cref{BFP98-Thm2-1-general}, since $A$ is totally positive, there exists a finite $2$-extension $L/K$ such that $a\in\N_{L/K}^\ast$ and $A_L$ is split.

If $\s$ is symplectic, then 
$\s_L$ is hyperbolic, whereby $a\in\N_{L/K}^\ast\subseteq\Hyp_2^\ast(A,\s)$.
Assume that $\s$ is orthogonal.
Then $(A_L,\s_L)\simeq \Ad(\vf)$ for a quadratic form $\vf$ over $L$. 
It follows from the hypothesis on $(A,\s)$ that $\Ad(\vf)$ is weakly hyperbolic and belongs to $\I^2$.
This means that $\vf\in\I^2_\tors L$.
By \Cref{In-vanish2ext},
 as $\I^3K(\sqrt{-1})=0$,  we have  $\I^3L(\sqrt{-1})=0$.
Hence, \Cref{In-torsfree-torsion-In-Hyp2} yields that $\Hyp_2^\ast(A_L,\s_L)=\Hyp_2^\ast(\vf)=\mg{L}$.
Therefore $a\in\N_{L/K}^\ast\subseteq \Hyp_2^\ast(A,\s)$.
\end{proof}

We set $\mg{(\sos{K})}=\sos{K}\cap\mg{K}$, which  is a subgroup of $\mg{K}$.

\begin{cor}\label{AP22}
Assume that $\I^3K(\sqrt{-1})=0$.
Let $(A,\s)$ be $K$-algebra with orthogonal involution with $(A,\s)\in\I^2$.
Then $$\sq{K}\subseteq\mg{(\sos{K})}\subseteq \Hyp_2^\ast(A,\s)\,.$$
Furthermore, if $A$ is totally positive, then
$$\G(A,\s)=\Hyp_2^\ast(A,\s)\,.$$
\end{cor}
\begin{proof}
Consider first $a\in \mg{(\sos{K})}$.
Set $L=K(\sqrt{-a})$. Then $L$ is nonreal and $\I^3L(\sqrt{-1})=0$, by \Cref{In-vanish2ext}.
Hence $\Hyp_2^\ast(A_{L},\s_{L})=\mg{L}$, by \Cref{KP}.
In particular $a=\N_{L/K}(\sqrt{-a})\subseteq \Hyp_2^\ast(A,\s)$.

Consider now an arbitrary element $a\in\G(A,\s)$.
Let $K'=K(\sqrt{-a})$.
Then $-1\in\G(A_L,\s_L)$, or in different terms, $2\times (A_L,\s_L)$ is hyperbolic. 
In particular, $(A_L,\s_L)$ is weakly hyperbolic.
Moreover, since $A$ is totally positive, so is $A_L$.
By \Cref{In-vanish2ext},
 as $\I^3K(\sqrt{-1})=0$,  we have  $\I^3K'(\sqrt{-1})=0$.
We obtain by \Cref{KP} that $\Hyp_2^\ast(A_{K'},\s_{K'})=\mg{K'}$.
In particular, we have $a=\N_{K'/K}(\sqrt{-a})\in\N_{K'/K}(\Hyp_2^\ast(A_{K'},\s_{K'}))\subseteq\Hyp_2^\ast(A,\s)$.
This shows that $\G(A,\s)\subseteq\Hyp_2^\ast(A,\s)$, and the opposite inclusion holds by \Cref{SNP}.
\end{proof}

When aiming to show that $\Hyp_2(\Phi)=\mg{K}$ for a $K$-algebra with involution $\Phi$ of a certain type under the assumption that  $\I^3K(\sqrt{-1})=0$, a reduction can be made to the case where $A$ (or another central simple $K$-algebra of exponent $2$)
 splits over $K(\sqrt{-1})$.
 This reduction method stems from \cite[Prop.~2.3]{CTS93} and \cite[Prop.~7.7]{BFP98}, and was further refined in \cite[Section 2]{KP08} and ingeniously used in \cite[Prop.~4.1 \& Lemma 6.5]{KP08}.
 We simplify this method and adapt it to the  hypothesis that $\I^3K(\sqrt{-1})=0$.
The following is a refinement of \cite[Prop.~2.3]{CTS93}.

\begin{lem}\label{CTS-lemma}
Let $\psi$ be a $4$-dimensional quadratic form over $K$ and $c\in\disc(\psi)$.
Then $\mg{K}\cap\G(\psi_{K(\sqrt{c})})$ is the union of the subgroups $\N_{K_1/K}^\ast\cdot\N_{K_2/K}^\ast$ where $K_i/K$ are field extensions with $[K_i:K]\leq 2$  such that $\psi_{K_i}$ is isotropic for $i\in\{1,2\}$.
\end{lem}
\begin{proof}
Note that $\psi_{K(\sqrt{c})}$ is a scaled $2$-fold Pfister form.
For any field extension $K'/K$ such that $\psi_{K'}$ is isotropic, we obtain that $\psi_{K'(c)}$ is hyperbolic,
whereby $\N_{K'/K}^\ast\subseteq\mg{K}\cap\N_{K'(\sqrt{c})/K(\sqrt{c})}^\ast\subseteq \mg{K}\cap \G(\psi_{K(\sqrt{c})})$, by \Cref{SNP}.

Consider now $a\in\mg{K}\cap\G(\psi_{K(\sqrt{c})})$.
Since the claim is invariant under scaling of $\psi$, we may assume that $1\in\D(\psi)$.
Then $\G(\psi_{K(\sqrt{c})})=\D(\psi_{K(\sqrt{c})})$.
Hence, there exist $x,y\in K^4$ such that $a=\psi(x+\sqrt{c}y)$.
It follows that $x$ and $y$ are orthogonal with respect to $\psi$ and $a=\psi(x)+c\psi(y)$.
We construct field extensions $K_i/K$ for $i\in\{1,2\}$ such that $\psi_{K_i}$ is isotropic, $[K_i:K]\leq 2$, and such that $a\in\N_{K_1/K}^\ast\cdot \N_{K_2/K}^\ast$ holds.

Assume first that $\psi(y)=0$. 
There exists a (regular) binary subform $\beta$ of $\psi$ which represents $1$ and $\psi(x)$. We take $d\in\disc(\beta)$ and set $K_1=K_2=K(\sqrt{d})$ to satisfy the claim.
Assume now that $\psi(x)=0$.
We then choose a (regular) binary subform $\beta$ of $\psi$ representing $1$ and $\psi(y)$ and then set $K_1=K(\sqrt{d})$ for some $d\in\disc(\beta)$.
Then $-acd=(-d)c^2\psi(y)\in\N_{K_1/K}^\ast$ and $\psi_{K_1}$ is isotropic.
Furthermore $\psi\simeq \beta\perp\beta'$ for a binary subform $\beta'$ of $\psi$ with $\disc(\beta')=cd\sq{K}$.
Letting $K_2=K(\sqrt{cd})$, we obtain that $\psi_{K_2}$ is isotropic and $-cd\in\N_{K_2/K}^\ast$, whereby 
$a\in\N_{K_1/K}^\ast\cdot\N_{K_2/K}^\ast$.
Assume finally that $\psi(x)\psi(y)\neq 0$.
We set $a_1=\psi(x)$ and $a_2=\psi(y)$, and obtain that $\la a_1,a_2\ra$ is a binary subform of $\psi$.
It follows that $\psi\simeq \la a_1,a_2,b,a_1a_2bc\ra$ for some $b\in\mg{K}$.
There further exists a (regular) binary subform $\beta$ of $\psi$ with $1,a_1\in\D(\beta)$. 
We set $K_1=K(\sqrt{d})$ for $d\in\disc(\beta)$, so that $a_1\in\N_{K_1/K}^\ast$.
Let $K_2=K(\sqrt{-a_1a_2c})$.
Since $a=a_1+ca_2$, we have $aa_1\in \N_{K_2/K}^\ast$, so $a\in\N_{K_1/K}^\ast\cdot\N_{K_2/K}^\ast$, and $\psi_{K_i}$ is isotropic for $i\in\{1,2\}$.
\end{proof}

\begin{lem}\label{-1reduction}
Assume that $\I^3K(\sqrt{-1})=0$.
Let $k\in\nat$.
Let $\vf$ be a quadratic form over $K$ with $\dim(\vf)=2k+2$ and $\disc(\vf)=\pm 1\cdot \sq{K}$.
Then there exists a set $\mc{L}$ of $2$-extensions of $K$ such that $[L:K]\leq 2^k$ and $\vf_{L(\sqrt{-1})}$ is hyperbolic for every $L\in\mc{L}$ and such that $\mg{K}=\la \bigcup_{L\in\mc{L}}\N_{L/K}^\ast\ra$.
\end{lem}
\begin{proof}
The proof is by induction on $k$.
If $k=0$, we take $\mc{L}=\{K\}$.
Assume that $k>0$. 
We fix $a_1,a_2,a_3\in\mg{K}$ such that $\la a_1,a_2,a_3\ra$ is a subform of $\vf$ and set $\psi =\la a_1,a_2,a_3,-a_1a_2a_3\ra$.
Then $\psi_{K(\sqrt{-1})}\in\I^2K(\sqrt{-1})$, and since $\I^3K(\sqrt{-1})=0$, it follows that $\G(\psi_{K(\sqrt{-1})})=\mg{K(\sqrt{-1})}$.

Consider $a\in\mg{K}$ arbitrary.
\Cref{CTS-lemma} applied with $c=-1$ yields that there exist field extensions $K_i/K$ with $[K_i:K]\leq 2$ such that  $\psi_{K_i}$ is isotropic and elements $a_i\in\mg{K_i}$ for $i\in\{1,2\}$ such that $a=\N_{K_1/K}(a_1)\cdot \N_{K_2/K}(a_2)$.

Fix $i\in\{1,2\}$.
Since $\vf_{K_i(\sqrt{-1})}$ is isotropic.
It follows that $\vf_{K_i}\simeq \vf'\perp \la d,d\ra$
for some $d\in\mg{K_i}$ and a quadratic form $\vf'\in\I^2K_i$ with $\dim(\vf')=2k$.
Since $\I^3K_i(\sqrt{-1})=0$, by \Cref{In-vanish2ext}, the induction hypothesis for $k-1$ applies to $\vf'$ over $K_i$.
We obtain a family $\mc{L}_i^{a}$ of $2$-extensions $L/K_i$ with $[L:K_i]\leq 2^{k-1}$ such that $\vf_{L(\sqrt{-1})}$ is hyperbolic and $a_i\in\la \bigcup_{L\in\mc{L}^a_i}\N_{L/K}^\ast\ra$.

We set $\mc{L}^a=\mc{L}_1^a\cup\mc{L}_2^a$.
Then $a\in\la \bigcup_{L\in\mc{L}^a}\N_{L/K}^\ast\ra$ and  $\mc{L}^a$ consists of $2$-extensions $L/K$ with $[L:K]\leq 2^k$ and such that $\vf_{L(\sqrt{-1})}$ is hyperbolic.
Constructing $\mc{L}^a$ for each $a\in\mg{K}$ in this way, the claim is satisfied for
$\mc{L}=\bigcup_{a\in\mg{K}}\mc{L}^a$.
\end{proof}

For symplectic involutions on $A$, the following recovers  \cite[Prop.~4.1]{KP08}, where $\cd_2(K(\sqrt{-1}))\leq 2$ is assumed. 
We follow the ideas of \cite[Prop.~4.1 \& Lemma 6.5]{KP08}.

\begin{thm}\label{classy-trivCliff}
    Assume that $\I^3K(\sqrt{-1})=0$.
    Let $(A,\s)$ be a weakly hyperbolic $K$-algebra with involution.
    If $\s$ is orthogonal, then assume that $(A,\s)\in\I^3$.
    Then $$\Hyp_2(A,\s)=\mg{K}\,.$$
\end{thm}

\begin{proof}
Consider a central simple $K$-algebra $D$ with $[D]\in\Br_2(K)$.
\Cref{Merkurjev} yields that $[D]=c(\vf)$ for a quadratic form $\vf\in\I^2K$.
By \Cref{-1reduction}, 
we can find a family $\mc{L}_D$ of finite $2$-extensions $L/K$ for which $\vf_{L(\sqrt{-1})}$ is hyperbolic and such that $\mg{K}=\la \bigcup_{L\in\mc{L}_D} \N_{L/K}^\ast\ra$.
Then, for $L\in\mc{L}_D$, $D_{L(\sqrt{-1})}$ is split  and $\I^3L(\sqrt{-1})=0$, by \Cref{In-vanish2ext}.
If now $\N_{L/K}^\ast\subseteq\Hyp_2(A,\s)$ holds for each $L\in\mc{L}_D$, then we conclude that $\Hyp_2(A,\s)=\mg{K}$.

If $\s$ is symplectic, we apply this with $D=A$,  associating a set of finite $2$-extensions $\mc{L}_A$ accordingly.
Let $L\in\mc{L}_A$. 
Then 
$(A_L,\s_L)\simeq \Ad(\vf)\otimes (Q,\can_Q)$
 for an $L$-quaternion algebra $Q$ and a quadratic form $\vf$ over $L$.
 Let $\pi$ denote the norm form of $Q$.
Since $(A_L,\s_L)$ is weakly hyperbolic, we have $\vf\otimes \pi\in\I_\tors K$.
If $\dim (\vf)$ is even, then $\vf\otimes\pi\in\I^3_\tors L$, and as $\I^3L(\sqrt{-1})=0$, we have $\I^3_\tors L=0$, by \cite[Cor.~35.27]{EKM08}, whence $\vf\otimes \pi$ is hyperbolic, whereby $(A,\s)_L$ is hyperbolic.
If $\dim(\vf)$ is odd, then $\pi\in\I_\tors^2K$, and we conclude that $\Nrd_{Q}^\ast=\mg{L}$.
In either case, we obtain that $\Hyp_2(A_L,\s_L)=\mg{L}$, and hence
 $\N_{L/K}^\ast\subseteq \Hyp_2(A,\s)$.

Assume now that $\s$ is orthogonal.
We consider the Rost invariant of $(A,\s)$.
This is a class in $H^3(K,\mu_4^{\otimes 2})/([A]\cup(\mg{K}))$.
Let $\xi\in H^3(K,\mu_4^{\otimes 2})$ be a representative of this class.
Since $\I^3K(\sqrt{-1})=0$, it follows by \Cref{2tors-cohom} that $\xi=(-1)\cup [D]$ for some central simple $K$-algebra $D$ with $[D]\in\Br_2(K)$.
We now associate the set $\mc{L}_D$ as above.
Consider $L\in\mc{L}_D$.
Then $D\sim (-1,c)_L$ for some $c\in\mg{L}$, and consequently $\xi=(-1)\cup (-1)\cup (c)$.
For every $P\in\X(K)$ such that $A_{K_P}$ is split,
 $\s_{K_P}$ is hyperbolic, whereby $\xi_{K_P}=0$ and thus $c\in P$.
This implies that $\xi_{K_P}=[A_{K_P}]\cup (c)$ for every $P\in\X(K)$.
We conclude by \Cref{2tors-cohom} that $\xi_L=[A_L]\cup(c)$.
Hence the Rost invariant of $(A_L,\s_L)$ is trivial.
Since $\I^3K(\sqrt{-1})=0$, we have $\I^3L(\sqrt{-1})=0$, by \Cref{In-vanish2ext}.
We conclude by \cite[Theorem 7.3]{BFP98} that $\s_L$ is hyperbolic.
Therefore $\N_{L/K}^\ast\subseteq\Hyp_2(A,\s)$.
\end{proof}

\begin{cor}
Assume that $\I^3K(\sqrt{-1})=0$.
Let $(A,\s)$ be a $K$-algebra with  orthogonal involution with $(A,\s)\in\I^3$.
    Then $$\G(A,\s)=\Hyp_2(A,\s).$$
\end{cor}
\begin{proof}
By \Cref{SNP}, we only need to show that $\G(A,\s)\subseteq\Hyp_2(A,\s)$.
Consider $a\in\G(A,\s)$.
Set $K'=K(\sqrt{-a})$.
Then $a=\N_{K'/K}(\sqrt{-a})$.
Moreover $-1\in\G(A_{K'},\s_{K'})$, so $2\times (A_{K'},\s_{K'})$ is hyperbolic. 
Hence $(A_{K'},\s_{K'})$ is weakly hyperbolic.
Since $\I^3K(\sqrt{-1})=0$, we have $\I^3K'(\sqrt{-1})=0$, by \Cref{In-vanish2ext}.
Therefore \Cref{classy-trivCliff} yields that $\Hyp_2(A_{K'},\s_{K'})=\mg{K'}$, whereby $a\in\N_{K'/K}(\Hyp_2(A_{K'},\s_{K'}))\subseteq\Hyp_2(A,\s)$.
\end{proof}

The following result extends \cite[Prop.~6.15]{KP08}.

\begin{thm}\label{KP08-6-15}
Assume that $\I^3K(\sqrt{-1})=0$.
Let $(A,\s)$ be a weakly hyperbolic $K$-algebra with orthogonal involution with $(A,\s)\in\I^2$. 
Let $C_1$ and $C_2$ be the two components of $\Cl(A,\s)$. 
Then $$\G(A,\s)=\Hyp_2(A,\s)=\Nrd_{C_1}^\ast\cdot\Nrd_{C_2}^\ast\,.$$
\end{thm}
\begin{proof}
We first show that $\G(A,\s)\subseteq \Nrd_{C_1}^\ast\cdot\Nrd_{C_2}^\ast$.
Consider $a\in\G(A,\s)$.
Then the $K$-algebra with involution $\Ad(\lla a\rra)\otimes (A,\s)$ is hyperbolic, so in particular its Rost invariant in $H^3(K,\mu_4^{\otimes 2})/([A]\cup(\mg{K}))$ is trivial.
By \cite[Lemma 6.9]{KP08}, this means that there exists $x\in\mg{K}$ such that
$$[C_1]\cup (a)= [A]\cup(x)\,.$$

Recall that, by \cite[Theorem~9.12]{KMRT98}, either
$2[C_1]=[A]$ or $[C_1]+[C_2]=[A]$.

Assume first that $2[C_1]=[A]$.
Then 
$[C_1]\cup(ax^{-2})=0$.
Hence, for every $P\in\X(K)$ with $\sign_P(\Tr_{C_1})<0$, since  $(C_1)_{K_P}\sim(-1,-1)_{K_P}$, we must have $a\in P$.
Therefore, in this case, \Cref{BFP98-Thm2-1-general} yields that $a\in\Nrd_{C_1}^\ast$, and we are done.

Assume now that $[C_1]+[C_2]=[A]$.
Then
$$[C_1]\cup (ax)= [C_2]\cup (x)\,.$$
Consider an ordering $P\in\X(K)$.
Since $\s_{K_P}$ is hyperbolic, $(C_1)_{K_P}$ or $(C_2)_{K_P}$ is split, whereby 
$$[(C_1)_{K_P}]\cup (ax)= [(C_2)_{K_P}]\cup (x)=0\,.$$
Therefore, on the one hand, either $(C_1)_{K_P}$ is split or $ax\in P$, and on the other hand, either $(C_2)_{K_P}$ is split or $x\in P$.
Having this for all $P\in\X(K)$, we conclude by \Cref{BFP98-Thm2-1-general} that $ax\in\Nrd_{C_1}^\ast$ and $x\in\Nrd_{C_2}^\ast$, whereby
$a=axx^{-1}\in\Nrd_{C_1}^\ast\cdot\Nrd_{C_2}^\ast$.

It remains to show that $\Nrd_{C_1}^\ast\cdot\Nrd_{C_2}^\ast\subseteq\Hyp_2(A,\s)$. Consider 
$i\in\{1,2\}$ and $a\in\Nrd_{C_i}^\ast$.
By \Cref{BFP98-Thm2-1-general}, there exists a finite $2$-extension $L/K$ such that $a\in\N_{L/K}^\ast$ and $(C_i)_L$ is split.
Then $(A_L,\s_L)$ has trivial Clifford invariant.
Since $\I^3K(\sqrt{-1})=0$, we have $\I^3L(\sqrt{-1})=0$, by \Cref{In-vanish2ext}.
Since $(A_L,\s_L)$ is weakly hyperbolic, we obtain by \Cref{classy-trivCliff} that $\Hyp_2(A_L,\s_L)=\mg{L}$.
Hence $a\in\N_{L/K}^\ast\subseteq\Hyp_2(A,\s)$.
\end{proof}

Following \cite{PW79}, $K$ is called an \emph{ED-field} if every $2$-dimensional torsion form over $K$ represents an element of $\mg{(\sos{K})}$.
This property is stable under $2$-extensions, by \cite[Cor.~3]{PW79}, and it implies the \emph{Strong Approximation Property} (\emph{SAP}) introduced in \cite{EL72}.

Assuming that $\I^3K(\sqrt{-1})=0$ and $K$ is an $ED$-field, it is shown in
\cite[Theorem 7.11]{KP08} that $\Hyp(A,\s)=\G(A,\s)=\mg{K}$ if $(A,\s)$ is weakly hyperbolic.
It is apparent from the discussion in \cite[Section 7]{KP08} that the $ED$-property is crucial for this conclusion. 
The following general example confirms this. 

\begin{ex}
Assume that $\I^3K(\sqrt{-1})=0$, but $K$ is not a $SAP$-field (and in particular not an $ED$-field). (For example, let $K$ be either $\rr(x,y)$ or $\rr(\!(x)\!)(\!(y)\!)$ where $x$ and $y$ are variables.)
Hence, there exist $x,y\in \mg{K}$ (e.g. the two variables in the examples) such that every multiple of the quadratic form $\la 1,x,y,-xy\ra$ over $K$ is anisotropic.
Consider then the $K$-quaternion algebras $Q_1=(-1,y)_K$, $Q_2=(-1,-y)_K$ and the $K$-algebra with orthogonal involution $(A,\s)_K=(Q_1,\can_{Q_1})\otimes (Q_2,\can_{Q_2})$.
Note that $\G(A,\s)=\Nrd_{Q_1}^\ast\cdot \Nrd_{Q_2}^\ast$ and $2\times (A,\s)$ is hyperbolic, whereby $(A,\s)$ is weakly hyperbolic.
On the other hand, since the form $\lla -1,y\rra\perp x\lla -1,-y\rra=2\times \la 1,x,y,-xy\ra$ over $K$ is anisotropic, we have that $-x\notin\Nrd_{Q_1}^\ast\cdot \Nrd_{Q_2}^\ast=\G(A,\s)$ (see also~\cite[Cor.~6.5]{KP08}), whereby $\G(A,\s)\neq \mg{K}$.
Hence, the condition in \Cref{KP} that $A$ is totally positive is not superfluous.
\end{ex}

\section{Involutions on index-$2$ algebras} 

Consider a $K$-quaternion algebra $Q$ and a $K$-algebra with orthogonal involution $(A,\s)$ such that $A\sim Q$.
Then $A_{K_Q}$ is split and $(A,\s)_{K_Q}\simeq \Ad(\vf)$ for a quadratic form $\vf$ defined over $K_Q$, which is unique up to similarity.
For $m\in\nat$ write $(A,\s)\in \I^m$ to indicate that 
$(A,\s)_{K_Q}\simeq \Ad(\vf)$ for a quadratic form $\vf$ over $K_Q$ whose class lies in $\I^mK_Q$.
If $A$ is split, then so is $Q$, whereby $K_Q/K$ is rational, and it follows that $(A,\s)\in\I^m$ holds if and only if 
$(A,\s)\simeq \Ad(\psi)$ for a quadratic form $\psi\in\I^mK$.

Starting from a $K$-algebra with orthogonal involution $(A,\s)$ with $\ind A\leq 2$, there is up to isomorphism a unique $K$-quaternion algebra $Q\sim A$, and hence it unnecessary to mention $Q$ for stating that $(A,\s)\in\I^m$.

Note that these conventions for being in $\I^m$ are compatible with those we took before for $m\in\{2,3\}$.

\begin{thm}\label{Berhuy-extended}
Let $n\in\nat$ be such that $\I^{n+1} K=0$.
Let $(A,\s)$ a $K$-algebra with orthogonal involution with $\ind A\leq 2$ and
$(A,\s)\in \I^n$.
Then $\s$ is hyperbolic if and only if $(A,\s)\in\I^{n+1}$.
\end{thm}
\begin{proof}
Let $Q$ be the $K$-quaternion algebra with $A\sim Q$ and $\vf$ a quadratic form over $K_Q$ such that $(A,\s)\simeq \Ad(\vf)$.
If $\s$ is hyperbolic, then $\s_{K_Q}$ is hyperbolic, whereby $\vf_{K_Q}=0\in\I^{n+1} K_Q$, so $(A,\s)\in\I^{n+1}$.
Assume that $\s$ is not hyperbolic.
By Dejaiffe's Theorem \cite{Dej01},
then $\s_{K_Q}$ is not hyperbolic, hence $\vf$ is not hyperbolic.
By the Arason-Pfister Hauptsatz \cite[Theorem 6.18]{EKM08}, there exists $m\in\nat$ such that $\vf\in\I^{m} K_Q\setminus \I^{m+1}K_Q$.
By \Cref{OVV}, then $e_m(\vf)\neq 0$ in $H^m(K,\mu_2)$.
It follows by \cite[Prop.~7 \& Prop.~9]{Ber07}
that $e_m([\vf_{K_Q}])$ lies in the image of the restriction homomorphism
$H^m(K,\mu_4^{\otimes m-1})\to H^m(K_Q,\mu_2)$.
Hence $H^m(K,\mu_4^{\otimes m-1})\neq 0$.
This implies that $\H^m(K,\mu_2)\neq 0$ and hence, by \Cref{OVV}, that $\I^m K\neq 0$.
Therefore $m\leq n$ and $\vf\notin \I^{n+1}K_Q$, so that $(A,\s)\notin\I^{n+1}$.
\end{proof}

\begin{prop}\label{trivial-rost-torsion}
    Let $n\in\nat$ be such that $\I^{n+1}K(\sqrt{-1})=0$.
    Let $k,d\in\nat$ and $\xi\in H^n(K,\mu_{2^{k}}^{\,\otimes d})$ such that $\xi_{K_P}=0$ for every $P\in\X(K)$.
    Then, for any $c\in\mg{K}$, there exists a finite $2$-extension $L/K$ with $c\in\N_{L/K}^\ast$ and $\xi_L=0$.
\end{prop}
\begin{proof}
The proof is by induction on $k$.
If $k=0$, then $\xi=0$ and the claim holds with $L=K$.
Assume now that $k>0$.
As $\xi\in H^n(K, \mu_{2^{k}}^{\,\otimes d})$, we have $2^{k-1}\cdot\xi\in H^n(K,\mu_2)$.
The short exact sequence $$1\rightarrow\mu_2\rightarrow \mu_{2^k}^{\otimes d} \stackrel{\cdot^{2}}{\rightarrow} \mu_{2^{k-1}}^{\otimes d}\rightarrow 1$$ induces an exact sequence of cohomology groups:
$$\cdots \rightarrow H^{n}(K,\mu_2)\rightarrow H^n(K,\mu_{2^{k}}^{\otimes d})\stackrel{\cdot 2^{k-1}}\rightarrow H^n(K,\mu_{2^{k-1}}^{\otimes d})\rightarrow \cdots$$
By the surjectivity of $e_n:\I^n K\to H^n(K,\mu_2)$, there exists $\phi\in \I^nK$ such that $e_n(\phi)=2^{k-1}\cdot\xi$.
For any ordering $P\in\X(K)$, we have $e_n(\phi_{K_P})=2^{k-1}\cdot\xi_{K_P}=0$, whereby $\sign_P(\phi)\in 2^{n+1}\zz$.
It follows by \cite[Cor.~35.27]{EKM08} together with \cite[Theorem 1]{Kru90} that $\phi\equiv\psi\bmod \I^{n+1}K$ for some quadratic form $\psi\in\I^n_\tors K$.

Consider some $c\in \mg{K}$. 
By \Cref{In-torsfree-torsion-In-Hyp2}, there exists a finite $2$-extension $K'/K$ such that $\psi_{K'}$ is hyperbolic and $c\in\N_{K'/K}^\ast$. 
Then $\phi_{K'}\in\I^{n+1}K'$, whereby $2^{k-1}\cdot\xi_{K'}=e_n(\phi_{K'})=0$.
Since $K'/K$ is a $2$-extension and $\I^{n+1}K(\sqrt{-1})=0$,  we have $\I^{n+1}K'(\sqrt{-1})=0$, by \Cref{In-vanish2ext}.
Fix $c'\in\mg{K'}$ with $c=\N_{K'/K}(c')$.
By the induction hypothesis, there exists a finite $2$-extension $L/K'$ such that $\xi'_L=0$ and $c'\in\N_{L/K'}^\ast$.
Then $L/K$ is a finite $2$-extension, $\xi_L=0$ and we have $c=\N_{K'/K}(c')\in\N_{K'/K}(\N_{L/K'}^\ast)=\N_{L/K}^\ast$.
\end{proof}

In the next two statements, we want to include the symplectic case.
Assume that $(A,\s)$ is a $K$-algebra with symplectic involution with $\ind A\leq 2$.
Let $Q$ be the $K$-quaternion algebra with $Q\sim A$.
We denote by $\can_Q$ the canonical involution on $Q$.
Then $(A,\s)\simeq \Ad(\vf)\otimes (Q,\can_Q)$
for a quadratic form $\vf$ over $K$.
Let $\pi$ denote the norm form of $Q$.
By Jacobson's Theorem (see e.g.~\cite[Theorem 4.6]{BU18}), the properties of the involution $\s$ and of the quadratic form $\pi\otimes \vf$ concerning hyperbolicity and similarity factors are the same, and this also over all field extensions.
In particular, we have $\G(A,\s)=\G(\pi\otimes\vf)$,
$\Hyp_2^\ast(A,\s)=\Hyp_2^\ast(\pi\otimes\vf)$, etc.
For $m\in\nat$, we write $(A,\s)\in\I^m$ if  $\pi\otimes\vf\in\I^m K$.

\begin{thm}\label{main}
Assume that $\I^{n+1}K(\sqrt{-1})=0$.
Let $(A,\s)$ be a weakly hyperbolic $K$-algebra with involution with $(A,\s)\in\I^n$. Assume that $A$ is totally positive and $\ind A\leq 2$. 
Then $$\G(A,\s)=\Hyp_2^\ast(A,\s)=\mg{K}\,.$$
\end{thm}
\begin{proof}
Let $Q$ denote the $K$-quaternion algebra Brauer equivalent to $A$.

Assume first that $\s$ is orthogonal.
Then $(A,\s )_{K_Q}\simeq \Ad(\vf)$ for some quadratic form $\vf\in\I^nK_Q$.
Consider $e_{n}(\vf)\in H^{n}(K_Q,\mu_2)$.
By \cite[Prop.~7 and Prop.~9]{Ber07}, this class lies in the image of the restriction map $H^{n}(K,\mu_4^{\otimes n-1})\to H^{n}(K_Q,\mu_4^{\otimes n-1})$.
Let $\xi\in H^{n}(K,\mu_4^{\otimes n-1})$ be a preimage of $e_{n}(\vf)$.

Consider an arbitrary ordering $P\in\X(K)$.
Since $A_{K_P}$ is split, the conic $\mc{C}_Q$ has a rational point over $K_P$, and hence $P$ extends to an ordering $P'\in\X(K_Q)$. Set $R=K_P$ and $R'=(K_Q)_{P'}$.
Since $\s_R$ is hyperbolic, so is $\s_{R'}$.
Hence $\vf_{R'}$ is hyperbolic, whereby $\xi_{R'}=0$.
Since the restriction map $\H^n(R,\mu_2)\to \H^n(R',\mu_2)$ is an isomorphism (between groups of order $2$), we conclude that $\xi_R=0$.
This shows that $\xi_{K_P}=0$ for every $P\in\X(K)$.

Consider $c\in\mg{K}$.
By \Cref{trivial-rost-torsion}, we have $c\in\N_{K'/K}^\ast$ for some finite $2$-extension $K'/K$ with $\xi_{K'}=0$.
Since $\I^{n+1}K=0$, we have $\I^{n+1}K'=0$, by \Cref{In-vanish2ext}.
Since $\xi_{K'}=0$, we have $e_{n}(\vf_{K'_Q})=0$, hence by \Cref{OVV}, we obtain that $\vf_{K'_Q}\in\I^{n+1} K'_Q$.
It follows by \Cref{Berhuy-extended} that $\s_{K'_Q}$ is hyperbolic.
By Dejaiffe's Theorem \cite{Dej01}, this implies that the involution $\s_{K'}$ is hyperbolic. 
Therefore $c\in\N_{K'/K}^\ast\subseteq \Hyp_2^\ast(A,\s)$.

Assume now that $\s$ is symplectic. 
Let $\pi$ denote the norm form of $Q$.
Then 
$(A,\s)\simeq \Ad(\vf)\otimes (Q,\can_Q)$ for a quadratic form $\vf$ over $K$, and 
the properties of the involution $\s$ correspond to those of the quadratic form $\pi\otimes\vf$.
In particular, $\Hyp_2^\ast(A,\s)=\Hyp_2^\ast(\vf\otimes \pi)$.
Since $(A,\s)\in\I^n$, we have $\pi\otimes\vf\in\I^nK$, and since $(A,\s)$ is weakly hyperbolic, we get that $\pi\otimes\vf\in\I^n_\tors K$.
Since $\I^{n+1}K(\sqrt{-1})=0$, we obtain by \Cref{In-torsfree-torsion-In-Hyp2} that 
$\Hyp_2^\ast(A,\s)=\Hyp_2^\ast(\vf\otimes \pi)=\mg{K}$.
\end{proof}

\begin{cor}\label{ind2-main}
    Assume that $\I^{n+1}K(\sqrt{-1})=0$.
    Let $(A,\s)$ be a $K$-algebra with involution with $(A,\s)\in\I^n$. Assume that $A$ is totally positive and $\ind A\leq 2$.
    Then $$\G(A,\s)=\Hyp_2^\ast(A,\s)\,.$$
    In particular ${\bf PSim}^+(A,\s)(K)/R=\{1\}$.
\end{cor}
\begin{proof}
    Consider $a\in\G(A,\s)$.
    Set $K'=K(\sqrt{-a})$.
    Then $A_{K'}$ is totally positive and $(A,\s)_{K'}$ is weakly hyperbolic. Hence $\Hyp_2^\ast(A_{K'},\s_{K'})=\mg{K'}$, by \Cref{main}, whereby $a\in\N_{K'/K}^\ast\subseteq\Hyp_2^\ast(A,\s)$.
    In view \Cref{SNP}, this shows the claim.
\end{proof}

\begin{rem}
    For any $n\geq 3$, the hypotheses of \Cref{main} imply that $A$ is split or $\deg A$ is a multiple of $4$.
    Indeed, if $A$ is nonsplit and $\deg A\equiv 2\bmod 4$, then $(A,\s)_{K(Q)}$ cannot have trivial Clifford invariant, by \Cref{G-Nrd-deg2mod4}. 
\end{rem}

\Cref{ind2-main} raises the question whether any of the two assumptions that $A$ is totally positive and $\ind A\leq 2$ can be relaxed or dropped. For $n\geq 3$, this presupposes a choice for the definition of belonging to $\I^n$.

\begin{qu}
     Assume that $\I^{n+1}K(\sqrt{-1})=0$.
    Let $(A,\s)$ be a $K$-algebra with involution with $(A,\s)\in\I^n$. Is ${\bf PSim}^+(A,\s)(K)/R=\{1\}$? 
\end{qu}

\section{Orthogonal involutions of trivial discriminant when $\I^4=0$} 

In \cite[Example 6.1]{PS17}, an example is given for an algebra with orthogonal involution $(A,\s)$ over $k=\qq_p(X)$ for which ${\bf PSim}^+(A,\s)(k)/R$ is nontrivial. 
Hence the underlying variety of the group is not rational. 
In that example, $\disc(\s)$ is non-trivial.
In \cite[Theorem 7.2]{PS15}, on the other hand, assuming that $\cd_2(K)\leq 3$, it is shown that ${\bf PSim}^+(A,\s)(K)/R$ is trivial 
 for any algebra $K$-algebra with orthogonal involution $(A,\s)$ of trivial discriminant and trivial Clifford invariant.
A similar result is achieved for symplectic involutions.
Here we revisit these results with milder hypotheses.

In particluar, we study ${\bf PSim}^+(A,\s)(K)/R$ when $(A, \s)$ is orthogonal of even degree and trivial discriminant, but with milder assumptions on the Clifford algebra. 
By \Cref{Mer:T1} and \Cref{G-G+}, triviality of ${\bf PSim}^+(A,\s)(K)/R$ amounts to the equality $$\G(A,\s)=\sq{K}\Hyp(A,\s)\,.$$
The results of this section establish even some refinements of this equality, under different hypotheses, assuming that $\I^4 K=0$.
We first consider the split case.
 
\begin{prop}\label{P:I40-quadI2GHyp2}
    Assume that $\I^4K=0$.
    Let $\varphi$ be a quadratic form over $K$ with $\vf\in \I^2 K$ and $\ind(\mc{C}(\varphi))\leq 4$.
    Then $\G(\vf)=\sq{K}\Hyp_2^\ast(\vf)$.
\end{prop}

\begin{proof}   
By \Cref{SNP}, we have $\sq{K}\Hyp_2^\ast(\vf)\subseteq\G(\vf)$, so only the other inclusion needs to be shown.

Since $\ind(\Cl(\vf))\leq 4$ and $[\Cl(\vf)]\in\Br_2(K)$, Albert's Theorem \cite[Theorem~16.1]{KMRT98} together with \cite[Prop.~16.3]{KMRT98} yield that $\mc{C}(\alpha)\sim\mc{C}(\varphi)$ for a $6$-dimensional quadratic form $\alpha\in\I^2K$. 
By \Cref{Merkurjev}, then $\varphi\perp-\alpha\in \I^3K$, and since $\I^4 K=0$, we conclude that $\G(\vf\perp-\alpha)=\mg{K}$.
Therefore $\G(\varphi)=\G(\alpha)$.

Consider $c\in\G(\vf)=\G(\alpha)$.
By \cite[Cor.~2.3]{EL73},
we can write $\alpha\simeq \beta_1\perp\beta_2\perp\beta_3$ with binary forms $\beta_1,\beta_2,\beta_3$ over $K$ such that $c\in\bigcap_{i=1}^3\G(\beta_i)$.
Let $d_i\in\disc(\beta_i)$ for $1\leq i\leq 3$.
Then $d_1d_2d_3\in\sq{K}$.
We set $L=K(\sqrt{d_1},\sqrt{d_2})$.
We obtain that $[L:K]\leq 4$ and 
$c\in \N_{K(\sqrt{d_1})/K}^\ast\cap\N_{K(\sqrt{d_2})/K}^\ast=\sq{K}\N_{L/K}^\ast$, by \cite[Lemma 3.3]{AB25}.
Furthermore $\alpha_L$ is hyperbolic, whereby
 $\vf_L\in\I^3 L$.
Since $\I^4K=0$, we have $\I^4 L=0$, by \Cref{In-vanish2ext}.
    It follows by \Cref{In-torsfree-torsion-In-Hyp2} that $\Hyp_2^\ast(\vf_L)=\mg{L}$.
    Therefore $c\in\sq{K}\N_{L/K}(\Hyp_2^\ast(\vf_L))\subseteq\sq{K}\Hyp_2^\ast(\vf)$.
 \end{proof}

The following statement extends \cite[Theorem~7.2]{PS15} by relaxing the condition on the Clifford invariant from trivial to index at most $2$.
Hence it provides
a partial positive answer to a question asked in \cite[p.~1028]{PS17}.

\begin{thm}\label{I4=0-Cliffindex2}
    Assume that $\I^4 K=0$.
    Let $(A,\s)$ be a $K$-algebra with orthogonal involution with $(A,\s)\in\I^2$. Let $C_1$ and $C_2$ be the components of $\Cl(A,\s)$. 
    Assume that $\ind(C_i)\leq 2$ for $i\in\{1,2\}$.
    Then 
    $$\G(A,\s)=\Nrd_{C_1}^\ast\cdot \Nrd_{C_2}^\ast=\Hyp_2(A,\s)\,.$$
\end{thm} 
\begin{proof}
    The inclusions $\Hyp_2(A,\s)\subseteq\G(A,\s)\subseteq \Nrd_{C_1}^\ast\cdot \Nrd_{C_2}^\ast$ hold by \Cref{SNP} and \Cref{G-in-Nrd}, respectively.
    It remains to show the inclusion $\Nrd_{C_i}^\ast\subseteq \Hyp_2(A,\s)$ for $i\in\{1,2\}$.
    Consider $i\in\{1,2\}$  and $c\in \Nrd_{C_i}^\ast$. 
    Since $\ind C_i\leq 2$, there exists a field extension $L/K$ with $[L:K]\leq 2$ such that $c\in\N_{L/K}^\ast$ and $(C_i)_{L}$ is split.
    Then $(A_L,\s_{L})$ has trivial Clifford invariant.
  Since $\I^4K=0$ and $[L:K]\leq 2$,
    we have  $\I^4 L=0$, by \Cref{In-vanish2ext}.
    Therefore $\Hyp_2(A_L,\s_{L})=\mg{L}$, by \Cref{main}, and we conclude that $c\in \N_{L/K}(\Hyp_2(A_L,\s_L))\subseteq\Hyp_2(A,\s)$.
\end{proof}

\begin{thm}\label{I4=0-deg2mod4Cliffindex4}
    Assume that $\I^4 K'=0$ for every field extension $K'/K$ of degree at most $4$.
    Let $(A,\s)$ be a $K$-algebra with orthogonal involution trivial discriminant and such that $\deg A\equiv 2\bmod 4$. 
    Let $C$ be a component of $\Cl(A,\s)$.
    Assume that $\ind C_{K_Q}\leq 4$ for the unique $K$-quaternion algebra $Q\sim A$. Then 
    $$\Hyp^\ast(A,\s)=\Nrd_C^\ast\quad\mbox{ and
    }\quad\G(A,\s)=\sq{K}\Hyp^\ast(A,\s)=\sq{K}\Nrd_{C}^\ast\,.$$
\end{thm} 
\begin{proof}
We have $\Hyp(A,\s)\subseteq\G(A,\s)\subseteq\sq{K}\Nrd_C^\ast$, by \Cref{SNP} and   \Cref{deg2mod4Cliffind4Nrd}, respectively.
Hence, to complete the proof, it suffices to show that $\Nrd_{C}^*\subseteq \Hyp^*(A,\s)$.

Consider $a\in \Nrd_C^\ast$. Let $D$ be the central $K$-division algebra with $D\sim C$.
Then $a\in\Nrd_D^\ast$, by \Cref{Nrd*}, so that $a\in \N_{L/K}^\ast$ for some maximal subfield $L$ of $D$.
Then $[L:K]=\deg D\leq \ind C\leq 4$ and $C_L$ is split.
Hence $\I^4 L=0$, by our hypothesis.
Since $\deg A\equiv 2\bmod 4$, we have $C^{\otimes 2}\sim A$, whereby $A_L$ is split 
and $(A_L,\s_L)\simeq\Ad(\vf)$ for a quadratic form $\vf$ over $L$.
Since $(A_L,\s_L)$ has trivial Clifford invariant,
it follows by \Cref{Merkurjev} that $\vf\in\I^3 L$.
Using \Cref{In-torsfree-torsion-In-Hyp2}, we obtain that $\mg{L}=\Hyp_2^\ast(\vf)=\Hyp_2^\ast(A_L,\s_L)$.
Therefore $a\in\N_{L/K}(\Hyp^\ast(A,\s))\subseteq \Hyp^\ast(A,\s)$.
\end{proof}

We recall the classical definition of the \emph{$u$-invariant of $K$}:
$$u(K)=\sup\{\dim(\vf)\mid \vf\mbox{ anisotropic quadratic form over }K\}\in\nat\cup\{\infty\}\,.$$
We refer to \cite[Chap.~XI, \S6]{Lam05} for an introduction to the study and the interest of this field invariant.

If $n\in\nat$ is such that $u(K)<2^n$, then every $n$-fold Pfister form over $K$ is hyperbolic, whereby $\I^n K=0$.

\begin{lem}\label{dim4}
Assume that $u(K)\leq 8$.
Let $\vf$ be a quadratic form over $K$ and $a\in\mg{K}$. 
    Then there exist $2$-extensions $L_1/K$ and $L_2/K$ of degree at most $4$
    such that $a\in \N_{L_1/K}^\ast\cdot \N_{L_2/K}^\ast$ and $\dim((\vf_{L_i})_\an)\leq 4$ for $i\in\{1,2\}$.
\end{lem}

\begin{proof}
Without loss of generality, we may assume that $\vf$ is anisotropic.
If $\dim(\vf)\leq 4$ then we may take $L_1=L_2=K$. 

Assume that $\dim(\vf)>4$. Then $\dim(\lla a\rra\otimes \vf)>8\geq u(K)$, whereby $\lla a \rra \otimes \vf$ is isotropic.
By \cite[Prop. 2.2]{EL73}, it follows that $\vf\simeq c\la 1, -d \ra \perp \psi$ for a quadratic form $\psi$ over $K$ and $c,d\in\mg{K}$ such that $\lla a,d\rra$ is hyperbolic.
As $\vf$ is anisotropic, we have $\dim(\vf)\leq u(K)\leq 8$.
Hence $\dim(\psi)\leq 6$.
Set $L=K(\sqrt{d})$. Then $a\in\N_{L/K}^\ast$ and $(\vf_L)_\an\simeq (\psi_L)_\an$. 
If $\dim((\psi_L)_\an)\leq 4$, we may take $L_1=L_2=L$.

Assume now that $\dim((\psi_L)_\an)>4$.
Then $\psi$ is anisotropic and of dimension $5$ or $6$.
It follows that $\dim(\lla a\rra\otimes\psi)>8\geq u(K)$, whereby $\lla a\rra\otimes \psi$ is isotropic.
By \cite[Prop.~2.2]{EL73}, we have $\psi\simeq c_1\la 1,-d_1\ra\perp\tau_1$ for a quadratic form $\tau_1$ over $K$ and $c_1,d_1\in\mg{K}$ such that $\lla a,d_1\rra$ is hyperbolic.
Since $a\in\N_{L/K}^\ast$, we fix $b\in \mg{L}$ with $a= \N_{L/K}(b)$.
Since $\N_{L/K}(b)=a\in\D(\lla d_1\rra_K)$, it follows by \cite[Chap.~VII, Theorem~5.10]{Lam05} that $b\in\mg{K}\cdot\D(\lla d_1\rra_L)$. 
We fix $b_1\in \D(\lla d_1\rra_L)$ and $b_2\in\mg{K}$ such that 
$b=b_1\cdot b_2$.
We set $L_1=L(\sqrt{d_1})$. 
Then $b_1\in \N_{L_1/L}^\ast$. 
Since $\dim(\lla b_2 \rra \otimes \psi)>8=u(K)$,
the quadratic form $\lla b_2\rra\otimes \psi$ over $K$ is isotropic.
Again by \cite[Prop.~2.2]{EL73}, there  exists a quadratic form $\tau_2$ over $K$ and $c_2,d_2\in\mg{K}$ such that $\psi\simeq c_2\la 1,-d_2\ra\perp \tau_2$ and $\lla b_2,d_2\rra$ is hyperbolic, whereby $b_2\in\D(\lla d_2\rra_{L})$.
We set $L_2=L(\sqrt{d_2})$. Then $b_2\in\N_{L_2/L}^\ast$, and we obtain that $$a=\N_{L/K}(b)=\N_{L/K}(b_1)\cdot \N_{L/K}(b_2)\in \N_{L_1/K}^\ast\cdot \N_{L_2/K}^\ast\,.$$
For $i\in\{1,2\}$, since 
$\vf\simeq c\la 1,-d\ra\perp c_i\la 1,-d_i\ra\perp \tau_i$ and $\vf$ is anisotropic, we have that 
 $\dim(\tau_i)=\dim(\vf)-4\leq u(K)-4\leq 4$, and as $d,d_i\in\sq{L}_i$, we obtain that $\dim((\vf_{L_i})_\an)\leq \dim(\tau_i)\leq 4$.
\end{proof}

The following statement generalises \cite[Theorem 6.1, Theorem 7.2]{PS15} and \cite[Theorem 5.1]{PS17}, where it is shown for function fields in one variable over $\qq_p$ for an odd prime $p$. 
By \cite[Theorem 3.4]{Lee13},
$u(K)=8$ holds when $K$ is a function field in one variable over $\qq_p$ for any prime $p$.
Hence this case is covered, but our result applies at the same time to arbitrary extensions of transcendence degree at most $3$ of an algebraically closed field.
Note that the hypothesis that $u(K)\leq 8$ implies that the index of elements in $\Br_2(K)$ is bounded by $8$, and no stronger bound on the index is required for the proof.

\begin{thm}\label{u8-trivial-clifford}
Assume that $u(K)\leq 8$. 
Let $(A,\s)$ be a $K$-algebra with involution with $(A,\s)\in\I^3$.
Then $$\Hyp_2(A,\s)=\G(A,\s)=\mg{K}\,.$$
\end{thm}
\begin{proof}
Since $A$ carries a $K$-linear involution, we have that $[A]\in\Br_2(K)$. 
Hence \Cref{Merkurjev} yields that $A\sim \Cl(\rho)$ for some anisotropic quadratic form $\rho \in \I^2K$. 
By \Cref{dim4}, $\mg{K}$ is generated by the subgroups $\N_{K'/K}^\ast$ where $K'/K$ is a finite $2$-extension with $\dim((\rho_{K'})_\an)\leq 4$, and consequently $\ind(A_{K'})\leq 2$.
Let $K'/K$ be such a finite $2$-extension. 
By \Cref{In-vanish2ext}, we have that $\I^4K'=0$.
In view of the hypotheses, we obtain by \Cref{main}  that $\Hyp_2^\ast(A_{K'},\s_{K'})=\mg{K'}$.
Therefore $\N_{K'/K}^\ast\subseteq \Hyp_2^\ast(A,\s)$.
We conclude that $\mg{K}= \Hyp_2(A,\s)$.
\end{proof}

One may ask whether the equality $\Hyp(A,\s)=\G(A,\s)$ still holds (whereby ${\bf PSim}^+(A,\s)(K)/R$ would be trivial) for $K$-algebras with orthogonal involution $(A,\s)\in\I^2$ (as opposed to $\I^3$)  when $u(K)\leq 8$.
Assuming that the $u$-invariant of all finite extensions of $K$ is bounded by $8$ (which is the case when $K$ is the function field of a $p$-adic curve), we can reduce this problem to the relation between $\G(A,\s)$ and the reduced norms of the components of $\Cl(A,\s)$.

\begin{cor}
Assume that $u(K')\leq 8$ for every finite field extension $K'/K$.
Let $(A,\s)$ be a $K$-algebra with orthogonal involution with $(A,\s)\in\I^2$.
Let $C_1$ and $C_2$ be the two components of $\Cl(A,\s)$.
Then $$\Hyp(A,\s)=\Nrd_{C_1}^\ast\cdot\Nrd_{C_2}^\ast\,.$$
\end{cor}
\begin{proof}
By \Cref{G-in-Nrd}, we have $\Hyp(A,\s)
\subseteq\Nrd_{C_1}^\ast\cdot\Nrd_{C_2}^\ast$.
For the opposite inclusion, consider $i\in\{1,2\}$ and $a\in \Nrd_{C_i}^\ast$. 
Hence, there exists a finite field extension $L/K$ such that $(C_i)_L$ is split and $a\in \N_{L/K}^*$. 
Then $u(L)\leq 8$ and $(A_L,\s_L)$ is an $L$-algebra with orthogonal involution of even degree with trivial discriminant and trivial Clifford invariant, so $\Hyp^\ast(A_L,\s_L)=\mg{L}$, by \Cref{u8-trivial-clifford}, and $a\in\N_{L/K}(\Hyp(A_L,\s_L))\subseteq \Hyp(A,\s)$.
This shows that $\Nrd_{C_i}^\ast\subseteq \Hyp^\ast(A,\s)$ for $i\in\{1,2\}$. 
Hence $\Nrd_{C_1}^*\cdot \Nrd_{C_2}^*\subseteq \Hyp(A,\s)$. 
\end{proof}

\bibliographystyle{plain}

\end{document}